\documentclass[11pt,a4paper]{article}
\usepackage[T1]{fontenc}
\usepackage[utf8]{inputenc}
\usepackage{lmodern}
\usepackage{mathtools,amssymb,amsmath,amsfonts,mathrsfs,textcomp,amsthm,latexsym}
\usepackage{tikz,enumitem,geometry,cite,bm,comment}
\usepackage{microtype,booktabs,xcolor,needspace}
\usepackage[colorlinks=true,linkcolor=blue!45!black,citecolor=blue!45!black,urlcolor=blue!45!black]{hyperref}
\hypersetup{pdftitle={Graph and Hypergraph Counterexamples to Mubayi's supersaturation Conjecture Beyond n to the power k-1}}
\numberwithin{equation}{section}
\newtheorem{theorem}{Theorem}[section]
\newtheorem{thm}[theorem]{Theorem}
\newtheorem{lemma}[theorem]{Lemma}
\newtheorem{lem}[theorem]{Lemma}
\newtheorem{proposition}[theorem]{Proposition}
\newtheorem{corollary}[theorem]{Corollary}
\newtheorem{conj}[theorem]{Conjecture}
\newtheorem{defi}[theorem]{Definition}
\newtheorem{claim}{Claim}
\theoremstyle{definition}

\DeclareMathOperator{\ex}{ex}
\DeclareMathOperator{\Aut}{Aut}
\DeclareMathOperator{\rank}{rk}
\newcommand{\N}{\mathcal{N}}

\newcommand{\join}{\vee}
\newcommand{\F}{\mathcal F}
\newcommand{\C}{\mathcal C}

\newcommand{\T}{\mathcal T}
\newcommand{\J}{\mathcal J}

\newcommand{\Ja}{\J_a}
\newcommand{\sq}{\sqcup}
\allowdisplaybreaks
\title{Graph and Hypergraph Counterexamples to Mubayi's supersaturation Conjecture Beyond $n^{k-1}$%
	 \thanks{ E-mail: 1519036064@qq.com (L. Zhang), qliu@fzu.edu.cn (Q. Liu), tqf9500@126.com (Q. Tang), yhong@fzu.edu.cn (Y. Hong)
}	\footnote{Mathematics Subject Classification [2020] 05C35, 05C65.}}
\author{Liping Zhang$^{a}$, Qinghai Liu$^{a,b}$\footnote{Corresponding author.},  Qinfei Tang$^{c}$, Yanmei Hong$^{b,d}$, 
\\[1mm]
{\small \it
	$^a$Center for Discrete Mathematics, Fuzhou University, Fuzhou, Fujian, 350108 China}\\
{\small \it
	$^b$Key Laboratory for Operations Research and Cybernetics of Fujian Universities}\\
{\small \it
	$^c$Department of Mathematics Research, Fujian Institute of Education, Fuzhou, China}\\
{\small \it
	$^d$School of Mathematics and Statistics, Fuzhou University, Fuzhou, Fujian, 350108 China}
}	\date{}
\begin{document}
\maketitle
\begin{abstract}
For a fixed $k$-graph $F$, let $h_F(n,q)$ be the minimum number of copies
of $F$ in an $n$-vertex $k$-graph with $\ex(n,F)+q$ edges.
Let $c(n,F)$ be the minimum number created by adding one edge to an
extremal $F$-free $k$-graph.
We construct stable counterexamples to Mubayi's supersaturation conjecture
for every $k\ge2$. The ratio $h_F(n,q)/(q c(n,F))$ tends to zero over
ranges extending beyond $n^{k-1}$.
For graphs, we give examples of every chromatic number at least three.
In particular, there is a graph $F$ for which this ratio tends to zero
for $1\le q=o(n^{4/3})$.
For $k\ge3$, we use connected linear expansions of
$G=(M_a\sq K_{1,b}\sq I_h)\vee K_{m_1,\ldots,m_{k-1}}$.
Under explicit parameter conditions, the ratio tends to zero for
$1\le q=o(n^{k-1+(b-1)/(a+b-1)})$ when $b\ge k+1$.
We prove the exact Tur\'an theorems for both constructions and determine
their unique extremal graphs and hypergraphs.

\medskip\noindent\emph{Keywords:} supersaturation; extremal graph;
linear hypergraph; graph expansion; stability.
\end{abstract}

\section{Introduction}
A hypergraph is a $k$-graph if every edge has size $k$.
We refer to $2$-graphs as graphs here. All graphs and hypergraphs in this paper are simple.
For a fixed $k$-graph $F$, the Tur\'an number $\ex(n,F)$ is the largest
number of edges in an $n$-vertex $F$-free $k$-graph.
A graph or hypergraph attaining this maximum is called extremal.
Write $\operatorname{EX}(n,F)$ for the family of extremal $F$-free $k$-graphs.
These questions originate from Tur\'an's theorem~\cite{g:Turan}.
It states that the unique extremal graph for $K_{d+1}$ is the balanced
complete $d$-partite graph $T_d(n)$.

Adding an edge to an extremal $F$-free graph or hypergraph creates a copy of $F$.
Supersaturation asks how many copies must occur when more edges are added.
We use both standard counting conventions.
For graphs, $\mathcal N_F(H)$ counts unlabelled copies.
It is the number of edge-preserving injections from $V(F)$ to $V(H)$,
divided by $|\Aut(F)|$.
For $k\ge3$, $N_F(H)$ counts the edge-preserving injections themselves.
The quantities $h_F(n,q)$, $t_F(n,q)$ and $c(n,F)$ use unlabelled copies
when $k=2$ and labelled copies when $k\ge3$.
Changing the convention multiplies all three quantities by the same
constant. In particular, it does not change $h_F(n,q)/(q c(n,F))$.

Let $h_F(n,q)$ be the minimum number of copies of $F$ among all
$n$-vertex $k$-graphs with $\ex(n,F)+q$ edges.
Let $t_F(n,q)$ be the corresponding minimum among the $k$-graphs obtained
by adding $q$ edges to a member of $\operatorname{EX}(n,F)$.
Set $c(n,F)=t_F(n,1)$.
Thus $h_F(n,q)\le t_F(n,q)$ and $q c(n,F)\le t_F(n,q)$.
Indeed, each added edge creates at least $c(n,F)$ copies containing no
other added edge. These collections of copies are disjoint.

Rademacher showed that, for large $n$, every graph with $n^2/4+1$ edges
contains at least $\lfloor n/2\rfloor$ triangles.
Erd\H{o}s~\cite{g:Erdos1,g:Erdos2} extended this result to a linear
range of additional edges.
Lov\'asz and Simonovits~\cite{g:LovaszSimonovits1975,g:LovaszSimonovits1976}
determined sharp ranges for triangles and more general cliques.
A graph $F$ is \emph{color-critical} if deleting some edge lowers its
chromatic number. Simonovits~\cite{g:Simonovits3} proved that, if
$\chi(F)=d+1$, then $T_d(n)$ is its unique extremal graph for all large $n$.
Mubayi~\cite{g:Mubayi1} proved that each color-critical graph $F$ has a
constant $\delta>0$ such that $h_F(n,q)\ge q c(n,F)$ for
$1\le q\le\delta n$ and all sufficiently large $n$.

A $k$-graph is \emph{$k$-partite} if its vertices can be partitioned into
$k$ sets so that every edge meets each set once.
We call $F$ \emph{stable} if its extremal $k$-graph $J_n$ is unique up to
isomorphism for all large $n$, and every $F$-free sequence $H_n$ with
$e(H_n)=\ex(n,F)-o(n^k)$ differs from $J_n$ in $o(n^k)$ edges after relabeling.
Mubayi~\cite{h:Mubayi2013} proposed the following conjecture.
\begin{conj}[Mubayi, 2013]\label{g:conj1}
Let $F$ be a non-$k$-partite stable $k$-graph with $k\ge2$.
For sufficiently large $n$ and any integer $q\ge1$,
$h_F(n,q)\ge q c(n,F)$.
\end{conj}
Ma and Yuan~\cite{h:MaYuan2025} constructed stable graph counterexamples
of every chromatic number at least four for $4\le q\le\delta_F n$.
Chen and Yuan~\cite{g:ChenYuan2026} constructed non-stable graphs for
which $h_F(n,1)=c(n,F)$ fails for all sufficiently large $n$.
Li, Liu, Liu and Wang~\cite{h:LLLW2026} gave stable counterexamples in
every uniformity. For every $k\ge2$ and $K>1$, they constructed a stable
non-$k$-partite $k$-graph $F$ such that
$h_F(n,q)\le K^{-1}q c(n,F)$ for $1\le q\le\delta_F n$.

We give graph and hypergraph counterexamples with an asymptotically
vanishing ratio $h_F(n,q)/(q c(n,F))$.
The graph construction gives examples of every chromatic number at least three.
The hypergraph construction gives connected linear examples for every $k\ge3$.
In both cases, the range of $q$ extends beyond $n^{k-1}$.
The proofs use exact extremal results and explicit competing constructions.

\noindent\textbf{Graph counterexamples.}
The first result concerns ordinary graphs.
\begin{thm}\label{g:res-2}
		For every integer $d\ge2$ and every integer $a> 3d/2-1$. There exists a  stable graph $F$ with
		$\chi(F)=d+1$ such that for all sufficiently
		large $n$ and all $q\ge 1$, 
		\[\frac{h_F(n,q)}{c(n,F)}
		=
		O_F\!\left(
		\frac{q^{a+4}}{n^{a+5}}
		\right),\quad
		\frac{h_F(n,q)}{qc(n,F)}
		=
		O_F\!\left(
		\frac{q^{a+3}}{n^{a+5}}
		\right).
		\]
	\end{thm}

Taking $d=2$ and $a=3$ gives a stable graph $F$ with $\chi(F)=3$.
There is a constant $\varepsilon>0$ such that
$h_F(n,q)<c(n,F)$ for $1\le q\le\varepsilon n^{8/7}$, and
$h_F(n,q)<q c(n,F)$ for $1\le q\le\varepsilon n^{4/3}$.

\begin{corollary}
		For every integer $d\ge2$ and every integer $a>3d/2-1$. There exist a  stable graph $F$ with
		$\chi(F)=d+1$ such that   the following holds. For all sufficiently
		large $n$,
		
		(i) if $ q=o (n^{a+5\over a+4})$, then $h_F(n,q)=o( c(n,F))$, and 
		
		(ii)  if $q=o (n^{a+5\over a+3 })$, then $h_F(n,q)=o (q\cdot c(n,F))$.	
	\end{corollary}

\noindent\textbf{Hypergraph counterexamples.}
We next consider expansions of graphs. For a graph $G$, its $k$-uniform
expansion $G^{+k}$ has one edge $\{x,y\}\cup Z_{xy}$ for each $xy\in E(G)$.
Each $Z_{xy}$ has size $k-2$. These sets are pairwise disjoint and avoid $V(G)$.
We call $G$ the core graph of this fixed expansion.
Every expansion is linear: distinct edges meet in at most one vertex.

For a positive integer $s$, a graph $Q$ is color-$s$-critical if there
are $s$ edges whose deletion lowers $\chi(Q)$, while deleting any $s-1$
vertices leaves $\chi(Q)$ unchanged. The $s$ edges must form an induced matching.
Simonovits~\cite[Theorem 2.2]{h:Simonovits1974} proved that, if $Q$ is
color-$s$-critical and $\chi(Q)=d+1\ge3$, then its unique extremal graph
for all sufficiently large $n$ is $K_{s-1}\vee T_d(n-s+1)$.
Here $Q_1\vee Q_2$ denotes the join of two vertex-disjoint graphs.

Let $\T^{(k)}(n)$ be the balanced complete $k$-partite $k$-graph, and
write $t^{(k)}(n)=e(\T^{(k)}(n))$.
For a fixed integer $a\ge0$, let $\J_a(n)$ be obtained from
$\T^{(k)}(n-a)$ by adding a disjoint $a$-set $S$ and every $k$-set that meets $S$.
Gerbner~\cite[Theorem 1.4(i)]{h:Gerbner2025} proved that
$\ex(n,Q^{+k})=e(\J_{s-1}(n))$ for all sufficiently large $n$ when
$Q$ is a disjoint union of $s$ connected graphs, each of chromatic number
$k+1$ and containing a color-critical edge.
Our exact theorem gives $\J_a(n)$ for connected cores of a different form.
The family includes color-$(a+1)$-critical graphs when $b=1$.
We also prove stability and determine the order of the single-edge count.

Let $M_a$ denote a matching with $a$ edges, and let $I_h$ denote an independent set of $h$ vertices, allowing $h=0$. Let
	\begin{equation}\label{h:eq:connected-parameters}
	G
	=(M_a\sq K_{1,b}\sq I_{h})\vee K_{m_1,\ldots,m_{k-1}}, \quad a\ge1, b\ge1, h\ge0,   m_i>0\ (1\le i\le k-1).
	\end{equation}
	All these parameters are fixed as $n$ tends to infinity. For disjoint graphs $G_1,G_2$, the symbol $\sq$ always denotes disjoint union.
	For  each integer $k\ge3$, define
	\begin{equation}\label{h:eq:main-constants}
	\begin{gathered}
	\alpha=\frac{k^{k-1}}{(k-1)!}-1,\ 
	\beta=\frac{k}{k-1}\left(\frac{(k-1)^{k-1}}{(k-1)!}-1\right),\
	D=\frac{k^{k-1}}{2(k-2)!}.
	\end{gathered}
	\end{equation}
	These constants depend only on $k$. 
	In particular,
	$
	\alpha-\beta=
	\frac{k\bigl(k^{k-2}-(k-1)^{k-2}\bigr)}{(k-1)!}
	+\frac1{k-1}>0.
	$
	Also $\beta>0$, since $(k-1)^{k-1}>(k-1)!$ for $k\ge3$.

	\begin{theorem}\label{h:thm:connected-exact}
		Let $k,b,h,a,m_1,\ldots,m_{k-1}$ be integers with $k\ge3$, $b\ge1$, and $h\ge0$ satisfying
		\begin{equation}\label{h:eq:admissible}
		a>(b-1)\max\left\{
		\frac{\beta+D}{\alpha-\beta},
		\frac{2D}{\alpha}
		\right\},\qquad
		m_i\ge a+1\quad(1\le i\le k-1).
		\end{equation}
		Then $G^{+k}$ from \eqref{h:eq:connected-parameters} is connected, linear,
		stable, and not $k$-partite. For all sufficiently large $n$, its
		unique extremal hypergraph is $\J_a(n)$. In particular,
		\[
		\ex(n,G^{+k})=\binom nk-\binom{n-a}k+t^{(k)}(n-a),\quad
		c(n,G^{+k})=\Theta_{G^{+k}}(n^{v(G^{+k})-a-k}).\]
	\end{theorem}

	When $b=1$, the graph $G$ in \eqref{h:eq:connected-parameters} is
	$
	(M_{a+1}\sqcup I_h)\vee K_{m_1,\ldots,m_{k-1}}.
	$
	In this case, the condition $a>(b-1)\max\left\{
	\frac{\beta+D}{\alpha-\beta},
	\frac{2D}{\alpha}
	\right\}$ is equivalent to $a\ge1$.

	\begin{theorem}\label{h:thm:connected-main}	Let $k,h,a,b,m_1,\ldots,m_{k-1}$ be integers with $k\ge3$ and $h\ge0$ satisfying \eqref{h:eq:admissible}. Suppose that $b\ge k+1$, and let $G$ be as in \eqref{h:eq:connected-parameters}.
		Let $\gamma=(b-1)/(a+b-1)$. There exists a constant $C_{G^{+k}}>0$ such that for all sufficiently large $n$
		and for all integers $1\le q\le n^{k-1+\gamma}$,
		\begin{equation}\label{h:eq:connected-ratio}
		\frac{h_{G^{+k}}(n,q)}{q c(n,G^{+k})}
		\le C_{G^{+k}}\left(
		\frac{n^{k-b}}q+
		\frac{q^{a+b-1}}{n^{(k-1)(a+b-1)+b-1}}\right).
		\end{equation}
		
	\end{theorem}

	\Needspace{8\baselineskip}\begin{corollary}\label{h:cor:optimized}
		Let $k,h,a,b,m_1,\ldots,m_{k-1}$ be integers with $k\ge3$ and $h\ge0$ satisfying \eqref{h:eq:admissible}. Suppose that $b\ge k+1$, and let $G$ be as in \eqref{h:eq:connected-parameters}. Let $\gamma=(b-1)/(a+b-1)$. The following results hold for all large $n$.
		\begin{enumerate}
			\item For every $K>1$, there exists $\delta>0$ such that
			$h_{G^{+k}}(n,q)\le K^{-1}q c(n,G^{+k})$ for all
			$1\le q\le\delta n^{k-1+\gamma}$;
			\item  $ h_{G^{+k}}(n,q)=o(q c(n,G^{+k}))$  for all
			$1\le q=o(n^{k-1+\gamma})$;
			\item For every fixed $0<c<\gamma$, there exists $\varepsilon>0$ such that $h_{G^{+k}}(n,q)=O_{G^{+k}}({1\over n^{\varepsilon}} q c(n,G^{+k}))$
			for all $1\le q\le n^{k-1+c}$.
		\end{enumerate}
	\end{corollary}

\noindent\textbf{Notation and proof outline.}
For a graph or hypergraph $H$, write $v(H)=|V(H)|$ and $e(H)=|E(H)|$.
For a graph $Q$, write $\chi(Q)$ for its chromatic number.
Write $[k]=\{1,\ldots,k\}$.
The symbols $K_a$ and $P_a$ denote the complete graph and a path on $a$ vertices.
For a vertex set $U$, write $Q[U]$ for the induced subgraph and $Q-U$
for the graph obtained by deleting $U$.
For a subgraph $Q'\subseteq Q$ and $v\notin V(Q')$, let $d(v,Q')$ be
the number of neighbors of $v$ in $Q'$.
The union $Q_1\cup Q_2$ has vertex set $V(Q_1)\cup V(Q_2)$ and edge set
$E(Q_1)\cup E(Q_2)$. For $Q'\subseteq Q$, the notation $Q-Q'$ means
$Q-V(Q')$.
For a partition into $k$ parts, a $k$-set is crossing if it meets every
part once; otherwise, it is noncrossing.

All parameters other than $n,q$ are fixed. All limits concern $n\to\infty$.
For nonnegative functions $u,v$, the relations $u(n)=O(v(n))$ and
$u(n)=\Omega(v(n))$ denote upper and lower bounds up to positive constant
factors for all sufficiently large $n$.
We write $u(n)=\Theta(v(n))$ when both bounds hold, and $u(n)=o(v(n))$
when $u(n)/v(n)\to0$.
Implicit constants may depend on the fixed forbidden graph and its parameters.
Subscripts indicate this dependence when needed.

Section~\ref{sec:graph-extremal} determines the unique extremal graphs
for the graph construction. It refines a stability partition and then
analyzes the neighborhoods of the exceptional vertices.
Section~\ref{sec:graph-supersaturation} counts copies in a competing
construction and proves the graph supersaturation bounds.
Section~\ref{h:sec:tools} collects the matching, embedding, structure
and stability results needed for hypergraphs.
Section~\ref{h:sec:connected} proves Theorem~\ref{h:thm:connected-exact}.
It establishes the single-edge count and local uniqueness, constructs
a bounded cover, and then proves exact uniqueness.
Section~\ref{h:sec:criterion} gives the hypergraph counting argument,
computes the forest rank, and proves Theorem~\ref{h:thm:connected-main}.
Section~\ref{h:sec:extensions} extends the exact extremal result to
connected subgraphs containing $a+1$ vertex-disjoint copies of $K_{k+1}$.

\section{Extremal graphs}\label{sec:graph-extremal}

	Define  $H(n,k-1,d) = K_{k-1} \vee T_d(n-k+1)$, which every vertex of $K_{k-1}$ is joined to all vertices of $T_d(n-k+1)$. The \emph{fork} graph (also known as the chair graph) is a 5-vertex tree obtained by subdividing one edge of the claw graph $K_{1,3}$. It contains exactly one vertex of degree three, one vertex of degree two and three leaves.

	\begin{thm}\label{g:thm1}
		Let \(a\) and \(d\) be positive integers satisfying \(a>3d/2\), and let $m_1, \dots, m_{d-1}$ be positive integers such that $m_i \ge a+2$ for each $i \in [d-1]$. Define $F = (M_a \cup B) \vee K(m_1, \dots, m_{d-1})$ where $B$ is a fork graph. Then for sufficiently large $n$,  $H(n, a+1, d)$ is the unique extremal graph for $F$.
	\end{thm}

	We prove Theorem~\ref{g:thm1} by first
	refining the stability partition to a cleaned structure and then analyzing the
	finitely many neighborhood profiles determined by the bounded exceptional set.
	The final subsection combines these structural statements to determine the extremal graphs. Throughout this section, unless otherwise stated, we fix
	\(F=(M_a\cup B)\vee K(m_1,\ldots,m_{d-1})\), where $B$ is the fork and
	$m_i\ge a+2$ for every $i\in[d-1]$.
	\subsection{Stability and cleaning}

	\begin{lem}\label{g:lem-H-free}
		Let $a, d\ge 2$ be integers, and let $m_1, \dots, m_{d-1}$ be positive integers such that $m_i \ge a+2$ for each $i \in [d-1]$. Define $F = (M_a \cup B) \vee K(m_1, \dots, m_{d-1})$ where $B$ is a fork graph. Then  $H=H(n, a+1, d)$ is $F$-free.
	\end{lem}
	
	\begin{proof} By construction, $H$ contains a vertex set $R$ with $|R|\leq a+1$ such that $\chi(H-R)\leq d$. Assume, for contradiction, that $H$ contains a copy $F_0$ of $F$. By the structure of $F$, decreasing its chromatic number from $d+1$ to $d$ requires removing at least $a+2$ vertices. Hence, for any set $S$ with $|S|\leq a+1$, we have $\chi(F_0-S)\geq d+1$. In particular, taking $S=V(F_0)\cap R$, we obtain $\chi(F_0-S)\geq d+1$. However, $F_0-S$ is a subgraph of $H-R$, so $\chi(F_0-S)\leq \chi(H-R)\leq d$. This contradiction shows that $H(n,a+1,d)$ is $F$-free. 	
	\end{proof}
	
	\begin{thm}[Erd\H{o}s--Simonovits Stability Theorem \cite{g:Simonovits3,g:Erdos1,g:Erdos2}]\label{g:thm-stable}  Let $F$ be a graph with $\chi(F)=d+1$. For every $\varepsilon > 0$, there exists a constant $\delta(F)> 0$  such that, if $n$ is sufficiently large and $G_n$ is an $F$-free graph with $n$-vertex satisfying that $e(G_n)\geq  (1-(1/d) -\delta){n\choose 2}$. Then $G_n$ can be obtained by omitting at most $\varepsilon n^2$ edges from a  $K(V_1,\cdots,V_d)$, and for every $ p\in[d]$, there is  $$\left|V_p|-\frac nd\right|<\varepsilon n, \quad e(G[V_p])<\varepsilon n^2.$$
	\end{thm}

	Let $G$ be an extremal graph for $F$.
	We shall use the Erd\H{o}s--Simonovits stability Theorem.  Since $\chi(F)=d+1$ and the construction
	in Lemma~\ref{g:lem-H-free} has $t_d(n)+O_F(n)$ edges, it implies that
	graph $G$ admits a $d$-partition with $o(n^2)$
	internal edges and class orders $n/d+o(n)$.  
	Choose a maximum $d$-cut $U_1\cup\cdots\cup U_d$, i.e., a partition
	maximizing 
	$P=\sum_{1\le i<j\le d}|U_i||U_j|$. Then for every $v\in U_i$ and $j\ne i$, we have
	\begin{align}\label{g:LC}
		d_G(v,U_i)\le d_G(v,U_j).                        
	\end{align}
	Otherwise, moving $v$ from $U_i$ to $U_j$ changes the number of actual
	cross-edges by
	$
	d_G(v,U_i)-d_G(v,U_j)>0,
	$
	contradicting the maximizing
	$P$.
	
	For a vertex pair $(x,y)$ with $x, y\in V_p$, if $xy\in E(G)$ then we   call this pair the \emph{internal edge}. 
	For  a vertex pair $(x,y)$ with $x\in V_p$ and $y\in V_q$ for $p\neq q$, if $xy\notin E(G)$ then we   call this pair the \emph{missing edge}. 
	Let 
	\[
	I=\sum_{i\in[d]} e(G[U_i]),\qquad
	M=\sum_{1\le i<j\le d}\bigl(|U_i||U_j|-e(U_i,U_j)\bigr).
	\]
	Therefore, $
	e(G)=P-M+I.
	$ The \(F\)-free construction from Lemma~\ref{g:lem-H-free} gives
	\[
	e(G)=\ex(n,F)\ge e(H(n,a+1,d))
	=t_d(n)+\frac{a+1}{d}n+O_F(1).
	\]
	Therefore
	\begin{align}\label{g:eq-IminusM}
		I-M=e(G)-P=(e(G)-t_d(n))+(t_d(n)-P)\ge\frac{a+1}{d}n+O_F(1).
	\end{align}

	
	For $x\in U_i$, let $\alpha(x)$ be the set of internal edges incident with
	$x$. For $j\neq i$, let $\beta_j(x)$ be the set of missing  edges incident with $x$ and a vertx of $U_j$ and define $\beta(x)=\bigcup _{i\in[d]}\beta_i(x)$.
	For a vertex set $U$, let
	$\alpha(U)=\bigcup_{x\in U}\alpha(x)$ and $\beta_i(U)=\bigcup_{x\in U}\beta(x)$. 
	
	The rest of this section turns stability into an exact statement.   We next refine this approximate
	structure. We first show that the
	number of internal edges and missing  edges is linear in \(n\), and
	then remove a bounded exceptional set to obtain a cleaned \(F\)-free
	graph. These reductions will provide the structural framework needed
	for the subsequent analysis of the local structure of the principal
	classes.

	\begin{claim}\label{g:lem-good-internal}
		Let 
		$0<\varepsilon<1/(3d|V(F)|)$.  For $i\in[d]$, let
		\[
		U_i^{\rm good}=\{x\in U_i:
		|\beta_j(x)|\le\varepsilon n
		\text{ for every }j\ne i\}.
		\]
		Then $G[U_i^{\rm good}]$ is $M_a\cup B$-free, and 
		$
		\sum_{i=1}^d e(G[U_i^{\rm good}])=O_F(n).
		$
	\end{claim}
	
	\begin{proof}
		Let $f_0=|V(M_a\cup B)|=2a+5$.  Every graph with more than
		$f_0N$ edges has a subgraph of minimum degree at least $f_0$, obtained
		by repeatedly deleting vertices of degree smaller than $f_0$.
		Consequently it contains every forest of order at most $f_0$, embedded
		greedily one leaf at a time.  Hence, if
		$e(G[U_i^{\rm good}])>f_0|U_i|$, then
		$G[U_i^{\rm good}]$ contains $M_a\cup B$.
		
		For this fixed copy, select the color classes of
		$K(m_1,\ldots,m_{d-1})$ successively in the other principal classes.
		At each step fewer than
		$v(F)\varepsilon n+O_F(1)<n/(2d)$ vertices are excluded by missing
		adjacencies, whereas stability gives $|U_j|=n/d+o(n)$.  Thus every
		selection can be made, producing a copy of $F$, a contradiction.
		Therefore $e(G[U_i^{\rm good}])\le f_0|U_i|$ for every $i$, and
		summing proves~\textup{(3)}.
	\end{proof}

	\begin{thm}[Ajtai, Koml{\'o}s, Simonovits  and Szemer\'edi~\cite{g:Ajtai}]\label{g:Theorem 2.9.} 
		Let $T$ be a tree with $l$ edges. If $l$ is  large enough, then  $\ex(n,T)\leq \frac{1}{2}(l-1)n$ for all integer $n$.
	\end{thm}
	\begin{lem}[Lemma 4\cite{g:Simonovits3}]	\label{g:lem-2.11} Let $m>0$ be a integer, and let $c>0$  be a constant. Then there exist an integer $m'\geq m$ and a constant $c'\leq c$ such that if a set $A$ with $|A|=n$ contains a family of  subsets $\{A_1,A_2,\cdots,A_{m'}\}$ with $|A_i|\geq cn$ for each $i\in [m']$, then there is a subset $\{A_1',A_2',\cdots,A_{m}'\}\subseteq\{A_1,A_2,\cdots,A_{m'}\}$ such that  $|\cap_{i\in[m]}A_i'|\geq c'n$.  
	\end{lem} 
	
	\begin{thm} \label{g:lem-high-bounded}  Let $K$ be  fixed, and let $G_n$ be a $n$-vertex $F$-free graph satisfying  
		$ e(G_n)\geq \ex(n,F)-Kn.$
		For every fixed $\varepsilon>0$, with
		respect to a maximum $d$-cut \(V(G)= U_1\cup\cdots\cup U_d\),  $e(G_n)$  has the following properties for sufficiently large  $n$.
		\begin{enumerate}[label=(\roman*)]
			\item $S:=|\{x\in G_n\mid \alpha(x)>\varepsilon n$ or $\beta(x)>\varepsilon n\}|=O_{\varepsilon}(1)$.
			\item If $\varepsilon< \frac{1}{2v(F)d^2}$, then $G[U_p\setminus S]$ is $M_a\cup B$-free for  all $p\in[d]$.
			\item  $I=O_F(n)$.
			$M=O_F(n)$.
		\end{enumerate}

	\end{thm}
	\begin{proof}	(i)
		Let $U_1\cup\cdots\cup U_d$ be the 	maximum $d$-cut  for $G_n$.
		For every $p\in[d]$, let
		\begin{align*}
			&B_p=\{x:x\in U_p, |\alpha(x)|\geq\varepsilon^2n\},\\	
			&C_p=\{x:x\in U_p, |\alpha(x)|<\varepsilon^2n, |\beta(x)|\geq\varepsilon n \}.	
		\end{align*}
		It is  sufficient to prove $B_p=O_F(1)$ and $C_p=O_F(1)$ since $\{x\mid \alpha(x)>\varepsilon n$ or $\beta(x)>\varepsilon n\}\subseteq B_p\cup C_p$. 
		
		
		We first prove that $B_p=O_F(1)$. Without loss of generality, suppose $B_1=\{x_1,\cdots,x_{m_0}\}$. Next, we will iteratively determine the  integers $m_0\geq m_1\geq\cdots \geq m_d$ and integers $c_1\geq\cdots\geq c_d$ such that for every $1\leq i\leq d$,  there are $m_i$ subsets $A_{i,1},\cdots,A_{i,m_i}$  of $A_i$, so that we can choose $m_{i+1}$ sets $A_{i,1}',\cdots,A_{i,m_{i+1}}'$ from $ \{A_{i,1},\cdots,A_{i,m_i}\}$ such that $|\cap_{1\leq j\leq i+1} A_{i,m_{j}}'|\geq c_i n$.

		Following \eqref{g:LC}, we can obtain the vertices $x_1,\cdots,x_{m_0}$ each  of which is joined to at least $\varepsilon^2n$ other vertices of $A_i$ for each $i\in[d]$.  Let $A_{i,j}$ denote the vertices of $A_j$ joined to $x_i$. Applying  Lemma \ref{g:lem-2.11} to $A_{1,1},\cdots, A_{1,m_0}$, we know that there are  $m_{1}$ vertices among $x_1,\cdots,x_{m_0}$ and a set $U_1\subseteq \cap_{i\in [m_0]} A_{1,i}$ such that $|U_1|\geq \varepsilon^2n$. Without loss of generality, let's assume these $m_{1}$ vertices are $x_1,\cdots,x_{m_{1}}$. Applying Lemma \ref{g:lem-2.11} again to $A_{2,1},\cdots, A_{2,m_{1}}$, then  there are  $m_{2}$ vertices of $x_1,\cdots,x_{m_{1}}$ and a set $U_2\subseteq \cap_{i\in [m_{1}]} A_{2,i}$ such that $|U_2|\geq \varepsilon^2n$. Repeating this operation,  we can obtain $m_d$ vertices ( assume that is $x_1,\cdots,x_{m_{d}}$) of $x_1,\cdots,x_{m_{d-1}}$ and $d$ subsets $U_1,\cdots, U_d$ such that  each of these $m_d$ vertices is joined to every vertex of $U_i$. On the other hand, the subgraph of $G_n$ spanned by the vertices $ U_1\cup\cdots\cup U_d$ contains a complete $d$-partite graph $K_{d}(v(F),\cdots,v(F))$ since the number of missing edges of $G_n$ is at most $o(n^2)$. Therefore, the subgraph of $G_n$ spanned by the vertices $\{x_1,\cdots,x_{m_d}\}\cup U_1\cup\cdots\cup U_d$ contains a copy of $F$ by setting $m_{d}=v(F)$, we get a contradiction.
		
		Next we prove that  $C_p=O_F(1)$. The following Claim is useful. 
		
		\begin{claim}\label{g:Claim 3} For any integer $c>0$, there is $\ex(n+c,F)\geq \ex(n,F)+c(1-\frac{1}{d}+o(1))n.$
		\end{claim}
		Let $X=\{x_1,x_2,\cdots,x_c\}$ be a vertex subset of $A_1$. Then $\sum_{i\in[c]} (\alpha(x_i)-\beta(x_i))=O_{\varepsilon}(n)$. Next, let $G_{n+c}$ be the graph obtained from $G_n$ by symmetrizing $c$ new vertices $\{y_1,\cdots, y_c\}$ onto $G(X)$, we know that  $G_{n+c}$ is $F$-free since Lemma 2.2. Hence,
		\begin{align*}\ex(n+c,F)\geq G_{n+c}&\geq G_{n}+c\left(1-\frac{1}{ d}-o(1)\right)n+\sum_{i\in[c]} (\alpha(x_i)-\beta(x_i))\\
			&\geq \ex(n,F)-Kn+c\left( 1-\frac{1}{d}-o(1)\right)n+O_{\varepsilon}(n)\\
			&= \ex(n,F)+c\left(1-\frac{1}{ d}+O_{\varepsilon}(1)\right) n.
		\end{align*}	
		This completes the proof.

		For any $y\in C_p$, we know that $d(y)\leq n-1-(\frac{n}{d}-o(1)n)+\varepsilon^2n-\varepsilon n<(1-\frac{1}{d}-\frac{1}{ 3}\varepsilon)n$ if $\varepsilon<1/3$. Let $y_1,\cdots,y_{c}$ be all vertices in $C_p$, and let $G_{n-c}=G_n-\{y_1,\cdots,y_{c}\}$. Therefore
		\begin{align*}
			\ex(n-c,F)\geq e(G_{n-c})\geq& e(G_n)-c\left(1-\frac{1}{d}-\frac{1}{ 3}\varepsilon\right)n\\
			&\geq \ex(n,F)-Kn-c\left( 1-\frac{1}{d}-\frac{1} {3}\varepsilon\right) n\\ &\geq\ex(n-c,F)+c\left( O_{\varepsilon}(1)+\frac{1} {3}\varepsilon\right)n -Kn,
		\end{align*}
		that means $c=O_F(1)$.

		(ii)  For each $p\in[d]$, $A_p\subseteq U_p^{\rm good}$. By Lemma \ref{g:lem-good-internal}, we have $G[A_p]$ is $M_a\cup B$-free. 
		
		
		(iii) Following (i), (ii) and Lemma \ref{g:Theorem 2.9.}, we have  $$|I|=\sum_{i=1}^de(G[U_p])<\sum_{i=1}^d(e_G(A_p)+e_G(B_p)+e_G(C_p))\leq d\cdot \ex(n,M_a\cup B)+O_F(1)=O_F(n).$$
		By \eqref{g:eq-IminusM}, $|M|=O_F(n)$.	
	\end{proof}
	
	\begin{defi}[Cleaned partition]
		\label{g:def:cleaned}
		Fix $\gamma>0$. A partition
		\(V(G)=W\cup U_1\cup\cdots\cup U_d\) of an $F$-free graph $G$ is called
		$\gamma$-cleaned
		if the following properties hold:
		\begin{enumerate}[label=(\roman*)]
			\item	$|W|=O_F(1)$, $|U_i|=n/d+o(n)$;
			\item for every vertex $v\in V(G)\setminus W$,  $|\alpha(v)|\le \gamma n$ and $|\beta(v)|\le \gamma n$;
			\item for every $i\in[d]$, the graph $G[U_i\setminus W]$ is $(M_a\cup B)$-free.
		\end{enumerate}	
		Moreover, we call $U_i$ principal classes for each $i\in[d]$.
	\end{defi}
	
	\begin{lem}[Cleaning Lemma]
		\label{g:lem-linear-cleaning}
		Fix a sufficiently small constant $\gamma>0$.  For every sufficiently large $n$ and every
		extremal $F$-free graph $G$, there is a $\gamma$-cleaned partition
		\(V(G)=W\cup U_1\cup\cdots\cup U_d\).
	\end{lem}
	\begin{proof}
		Start with a maximum $d$-cut $V_1\cup\cdots\cup V_d$. If $|M|=O_F(n)$, then  $|U_i|=n/d+o(n)$.
		Let
		\(W:=\{x:|\alpha(x)|>\gamma n\text{ or }|\beta(x)|>\gamma n\}\), and let  $U_i:=V_i\setminus W$.  
		By
		Lemma~\ref{g:lem-high-bounded},  \(V(G)=W\cup U_1\cup\cdots\cup U_d\)  is 
		$\gamma$-cleaned. 
	\end{proof}

	\subsection{Profile classes and the universal core}
	After obtaining a cleaned \(F\)-free graph in the previous subsection,
	we further analyze the local structure of its principal classes. We
	classify the vertices according to their neighborhoods in the bounded
	exceptional set $W$ and study the resulting profile classes.
	The main goal
	of this subsection is to show that the cleaned graph admits a highly
	restricted finite profile structure. These structural
	properties will be used in the final proof of Theorem~\ref{g:res-2}.
	\begin{lem}\label{g:lem-fork-bound}
		If a graph $H$ is fork-free, then
		$
		e(H)\le\frac32v(H).
		$
		Equality holds only when every component is $K_4$.  With arbitrary
		order of $H$, the exact maximum is attained by
		$\lfloor v(H)/4\rfloor K_4\cup K_r$, where
		$r\equiv v(H)\pmod4$ and $0\le r<4$.
	\end{lem}
	
	\begin{proof}
		We prove the bounds for each component. Let $C$ be a component of $H$.
		
		If $\Delta(C)\le 2$, then $C$ is a path or a cycle, so $e(C)\le |V(C)|\le \frac{5}{4}|V(C)|$.
		If $\Delta(C)\ge 3$ and $|V(C)|\le 4$, a direct check gives $e(C)\le 6 = \frac{3}{2}|V(C)|$, and $e(C)\le 5 = \frac{5}{4}|V(C)|$ whenever $C\not\cong K_4$.
		
		Now assume $\Delta(C)\ge 3$ and $|V(C)|\ge 5$. Let $u$ be a vertex of maximum degree with neighbors $u_1,\dots,u_r$ ($r\ge3$). If there exists a vertex $v$ adjacent to, say, $u_1$ with $v\notin\{u,u_2,\dots,u_r\}$, then the induced subgraph on $\{u,u_1,u_2,u_3,v\}$ contains the edges $uu_1,uu_2,uu_3$ and $u_1v$, which form a fork, a contradiction. Hence $V(C)=\{u,u_1,\dots,u_r\}$, so $r\ge4$. Moreover, no two leaves $u_i,u_j$ ($i,j\ge2$) can be adjacent, otherwise a fork appears. Thus $C$ is a star $K_{1,r}$ and $e(C)=|V(C)|-1$. For $|V(C)|\ge5$ we have $|V(C)|-1\le\frac{5}{4}|V(C)|$.
		
		In all cases, $e(C)\le\frac{3}{2}|V(C)|$ holds for every component,  with equality only for
		$K_4$. Summing over components yields $e(H)\le\frac{3}{2}|V(H)|$. For the exact remainder statement, write $v(H)=4s+r$, $0\le r<4$.
		Every component of order at least five which is not $K_4$ has at most
		its order many edges, and every noncomplete component of order four
		has at most five edges.  Replacing such components by copies of $K_4$
		and one clique $K_r$ never decreases the edge count and is strict
		unless the original components already have this form.  Hence the
		unique extremal component multiset is $sK_4\cup K_r$, as asserted.
	\end{proof}
	
	\begin{lem}\label{g:lem-dichotomy}
		If $A$ is $(M_a\cup B)$-free, then either $A$ is fork-free or there is
		$X\subseteq V(A)$ with $|X|\le2a+3$ such that $A-X$ is independent.
	\end{lem}
	
	\begin{proof}
		Fix a fork $B_0\subseteq A$.  Then $A-V(B_0)$ has matching number at
		most $a-1$.  The endpoints of a maximal matching form a vertex cover
		of order at most $2a-2$.  Together with $V(B_0)$ they give $X$.
	\end{proof}

	\begin{defi}[Profile classes and positive-density types]
		\label{g:def-profile-class}	
		Let  $V(G)=W\cup U_1\cup\cdots\cup U_d$ be a cleaned partition for graph $G$.	
		For each \(i\in[d]\) and \(Y\subseteq W\), define the profile class
		\[
		U_{i,Y}:=\{u\in U_i:N_G(u)\cap W=Y\}.
		\]
		We call \(U_{i,Y}\) a positive-density type if
		\(|U_{i,Y}|=\Omega_F(n)\). Otherwise, we call it a negligible type.
		The class \(U_{i,\emptyset}\) is referred to as the empty type of
		\(U_i\).
	\end{defi}
	\begin{lem}
		\label{g:lem-profile-reduction}
		Fix a sufficiently small constant $\gamma>0$.	Let $G$ be a $F$-free graph with $\gamma$- cleaned partition  $V(G)=W\cup U_1\cup\cdots\cup U_d$. After enlarging \(W\) by at
		most \(O_F(1)\) vertices (if necessary), and redefining the profile classes,
		the following hold:
		\begin{enumerate}[label=(\roman*)]
			\item for every $i\in[d]$, the graph $G[U_{i,\varnothing}]$ is fork-free;
			\item 	for every vertex  $v\in V(G)\setminus W$, $|\alpha  (v)|\le \gamma n$ and $|\beta (v)|\le \gamma n$;
			\item  every  profile class $U_{i,Y}$ with non-empty $Y\subseteq W$ is incident with only $o(n)$ internal edges;
			
			\item \(\sum_{i=1}^d e(G[U_i])\le\frac32\sum_{i=1}^d|U_{i,\varnothing}|+o(n)\).
		\end{enumerate}
	\end{lem}
	\begin{proof}

		\textup{(i)} 
		Apply Lemma~\ref{g:lem-dichotomy} to every $G[U_{i,\emptyset}]$. In the non-fork-free branch,
		there exists a vertex set
		$C_i\subseteq U_i\setminus W$ with
		$	|C_i|=O_F(1)
		$
		such that
		$	G[U_i\setminus (W\cup C_i)]
		$
		is independent. Enlarging the exceptional set $W$ by all such sets
		$C_i$, which changes $|W|$ only by $O_F(1)$.
		The remaining graph
		in that principal class is independent. This proves~(i) after the profiles are
		redefined. Clearly,~(ii) is from  the cleaning property; enlarging $W$ by $O_F(1)$ vertices does not affect it.

		~\textup{(iii)} 
		Fix $\delta>0$, and choose the cleaning
		parameter $\gamma\ll\delta$ sufficiently small. Suppose, for contradiction, 
		there exist an index $i\in[d]$ and non-empty $Y\subseteq W$, $U_{i,Y}$  is incident with $
		\delta n
		$ internal edges of $G[U_i]$.  
		By (ii),  $|\alpha  (v)|\le \gamma n$ for every vertex  $v\in V(G)\setminus W$, the set of internal edges  above contains a matching of size
		at least
		$\frac{\delta n}{2\gamma n+1}.$
		Taking $\gamma$ sufficiently small, this matching contains more than
		$a+4$
		pairwise vertex-disjoint edges.		Let
		$
		x_1y_1,\ldots,x_sy_s
		$
		be such a matching, where
		$
		s=a+4.
		$
		By the definition of $U_{i,Y}$,	there is 	vertex
		$
		w\in Y
		$
		such that  every vertex of $U_{i,Y}$ is adjacent
		to $w$.	
		We may choose one edge, say $x_1y_1$, together with two further
		vertices
		$
		z_1,z_2\in U_{i,Y}
		$
		disjoint from it. Then
		$
		w,x_1,y_1,z_1,z_2
		$
		contains a copy of the fork $B$, with $w$ as its center.
		Next choose $a$ edges from the remaining matching edges, say
		$	x_{j_1}y_{j_1},\ldots,x_{j_a}y_{j_a},
		$
		which are disjoint from the above copy of $B$. These edges form a copy
		of the matching $M_a$.

		It remains to embed 
		$
		K(m_1,\ldots,m_{d-1}).
		$
		By (ii), ,  $|\beta (v)|\le \gamma n$ for every vertex  $v\in V(G)\setminus W$.
		Fixing the vertices already selected for
		$M_a\cup B$, and taking $\gamma$ sufficiently small. By  Lemma \ref{g:lem-good-internal},  we can greedily obtain a copy of
		$
		F=(M_a\cup B)\vee K(m_1,\ldots,m_{d-1}),
		$
		contradicting the assumption that $G$ is $F$-free.
		Therefore the number of internal edges incident with 
		non empty types is less than
		$
		\delta n.
		$
		Since $\delta>0$ is arbitrary, this contribution is $o(n)$.	
		
		(iv) Let $D_i$ be the union of the negligible nonempty profiles in $U_i$.
		Since \(|W|=O_F(1)\), the number of profile classes is at most
		\(2^{|W|}=O_F(1)\). Hence
		$
		|D_i|=O_F( n).
		$
		After the bounded enlargement of \(W\) supplied by
		Lemma~\ref{g:lem-dichotomy}, each principal class is in one of the
		following two situations. It belongs to the independent branch. Otherwise
		it belongs to the fork-free branch, then the induced graph
		$
		G[D_i\cup U_{i,\emptyset}]
		$
		is fork-free. Hence Lemma~\ref{g:lem-fork-bound} gives
		\[
		e\bigl(G[D_i\cup U_{i,\emptyset}]\bigr)
		\leq
		\frac32\bigl(|D_i|+|U_{i,\emptyset}|\bigr)
		=
		\frac32|U_{i,\emptyset}|+O_F( n).
		\]
		Edges from $D_i$ to positive-density nonempty profiles are already included in
		the $o(n)$ contribution from~(iii). Summing over $i\in[d]$ proves~(iv). 
	\end{proof}

	\begin{lem}
		\label{g:lem-profile}
		Assume $a+1>  3d/2$. 	Let $G$ be a $F$-free graph with  cleaned partition  $V(G)=W\cup U_1\cup\cdots\cup U_d$. Let
		\(U_{i,Y}\) be the profile classes given by Lemma~\ref{g:lem-profile-reduction}.
		Then 
		\[
		e(G)\leq t_d(n)+\frac{a+1}{d}n+o(n).
		\]
		Moreover, equality in the above bound implies that there exists
		a set
		$C\subseteq W, |C|=a+1,
		$ such that for every positive-density type
		$U_{i,Y}$,
		$
		C\subseteq Y.
		$
		And every vertex \(w\in W\setminus C\) satisfies the
		following property: there exists a unique \(i\in[d]\) such that
		\(w\notin Y\) for every positive-density type \(U_{i,Y}\), while
		\(w\in Y\) for every positive-density type \(U_{j,Y}\) with
		\(j\neq i\).	
	\end{lem}
	\begin{proof}
		For every $i\in[d]$, choose a set $Y_i\subseteq W$ randomly according
		to the distribution induced by the  classes of $U_i$, that is,
		$
		\mathbb P(Y_i=Y)=\frac{|U_{i,Y}|}{|U_i|}.
		$
		Since $|W|=O_F(1)$, the number of possible subset $Y$ is bounded by a
		constant. For $w\in W$, define
		$p_i(w):=\mathbb P(w\in Y_i).$
		
		We first estimate the contribution of the vertices in $W$. By the
		definition of $p_i(w)$, the expected number of principal classes in
		which $w$ has a positive-density neighbourhood is
		$\sum_{i=1}^{d}p_i(w).$
		Thus
		$
		\sum_{i=1}^{d}p_i(w)
		\leq
		d-1+\prod_{i=1}^{d}p_i(w).
		$
		Indeed, if $w$ does not belong to 
		$\bigcap_{i=1}^{d}Y_i$, then it can occur in at most $d-1$ of the
		sets $Y_i$, while the remaining contribution is exactly the probability
		that $w$ belongs to every $Y_i$.
		Summing  over all $w\in W$ gives
		\begin{align}\label{g:2.12}
			\sum_{i=1}^{d}\mathbb E|Y_i|=	\sum_{w\in W}\sum_{i=1}^{d}p_i(w)
			\leq
			(d-1)|W|
			+
			\mathbb E\left|\bigcap_{i=1}^{d}Y_i\right|.
		\end{align}

		We next bound the last term in \eqref{g:2.12}. Every
		positive-density type is determined by one of the sets $Y_i$. If for
		some choice of positive-density types we have
		$	\left|\bigcap_{i=1}^{d}Y_i\right|\geq a+2,$
		then choosing $a+2$ vertices from this intersection and using the
		corresponding positive-density type  would yield a copy of
		$M_a\cup B$. Indeed, two vertices from the intersection together with
		an internal edge in one positive-density type form the fork,
		while $a$ further vertices provide the matching. Since $U_i$ have linear size and  every vertex of  $ V(G)\setminus W$ have only $o(n)$
		missing edges in every other  class, the remaining
		vertices of
		$
		K(m_1,\ldots,m_{d-1})
		$
		can be greedily embedded by Lemma \ref{g:lem-good-internal}. This produces a copy of
		$
		F=(M_a\cup B)\vee K(m_1,\ldots,m_{d-1}),
		$
		a contradiction. Hence
		$
		\left|\bigcap_{i=1}^{d}Y_i\right|\leq a+1,
		$.

		Let
		$
		z_i:=\mathbb P(Y_i=\emptyset)
		$
		be the probability of the empty set in $U_i$. If one of the
		selected sets $Y_i$ is empty, then $
		\bigcap_{i=1}^{d}Y_i=\emptyset
		$.
		Otherwise, 	$
		\left|\bigcap_{i=1}^{d}Y_i\right|\leq a+1
		$ by the above
		argument. Hence
		\[
		\mathbb E\left|\bigcap_{i=1}^{d}Y_i\right|
		\leq
		(a+1)\prod_{i=1}^{d}(1-z_i).\] 
		Combining the  inequality \eqref{g:2.12} , we obtain
		$
		\sum_{i=1}^{d}\mathbb E|Y_i|
		\leq
		(d-1)|W|+(a+1)\prod_{i=1}^{d}(1-z_i) .
		$

		We now estimate the number of edges of $G$. Write
		\[
		e(G)=e(U_1,\ldots,U_d)+e(W,U_1\cup\cdots\cup U_d)
		+\sum_{i=1}^{d}e(U_i).
		\]
		By Lemma~\ref{g:lem-profile-reduction} (iii), every vertex of  $ V(G)\setminus W$
		has $o(n)$ missing edges in each other principal class. Hence
		$$
		e(U_1,\ldots,U_d)=t_d(n-|W|)+o(n)=
		t_d(n)-\frac{d-1}{d}|W|n+o(n).
		$$
		For the second term, 
		\[
		e(W,U_1\cup\cdots\cup U_d)
		=
		\frac nd\sum_{w\in W}\sum_{i=1}^{d}p_i(w)+o(n)
		\leq \frac nd\left((d-1)|W|+(a+1)\prod_{i=1}^{d}(1-z_i)+o(1)\right).
		\]
		It remains to estimate the internal edges. By
		Lemma~\ref{g:lem-profile-reduction} (iv), 
		\[
		\sum_{i=1}^{d}e(U_i)
		\leq
		\frac32\sum_{i=1}^{d}|U_{i,\emptyset}|+o(n)
		=
		\frac32\left(\sum_{i=1}^{d}z_i\right)\frac nd+o(n).
		\]	On the other hand, 
		$1-\prod_{i=1}^{d}(1-z_i)
		\geq
		\max_{i\in[d]} z_i
		\geq
		\frac1d\sum_{i=1}^{d}z_i $. Since 
		$
		a+1>\frac{3d}{2},
		$
		we obtain
		\begin{align}\label{g:2.17}
			(a+1)\prod_{i=1}^{d}(1-z_i)
			+\frac32\sum_{i=1}^{d}z_i
			\le
			a+1 .
		\end{align} 	
		Then
		\[
		e(W,U_1\cup\cdots\cup U_d)+\sum_{i=1}^{d}e(U_i)
		\leq \frac{d-1}{d}|W|n+
		\frac{a+1}{d}n+o(n).
		\]
		That means the possible contribution of the empty $Y_i$'s is already
		absorbed by the loss of the contribution from $W$.	
		Combining the above estimates yields
		$
		e(G)
		\leq
		t_d(n)+\frac{a+1}{d}n+o(n).
		$

		Finally, we consider the equality case. Suppose that
		$
		e(G)=t_d(n)+\frac{a+1}{d}n+o(n).
		$
		Then all inequalities above must be asymptotically tight. In particular,
		equality must hold in \eqref{g:2.17}, which is possible only if $\sum_{i}^dz_i=0$ and $
		|\bigcap_{i=1}^{d}Y_i|=a+1
		$. Hence there exists a set
		$
		C\subseteq W$ with $|C|=a+1,
		$
		contained in every positive-density type.	
		Therefore, for every positive-density type $U_{i,Y}$ we have
		$
		C\subseteq Y.
		$
		By the definition of $U_{i,Y}$, every vertex contained in a
		positive-density sest is adjacent to every vertex of $C$.
	\end{proof}
	
	We now fix the core $C$ obtained from the equality case and analyze the remaining vertices  $V(G)\setminus C$.
	\begin{lem}	\label{g:lem:exact-charging} 	Let $G$ be a $F$-free graph with  cleaned partition  $V(G)=W\cup U_1\cup\cdots\cup U_d$.
		Suppose that there exists a set
		$
		C\subseteq W$ with $ |C|=a+1$
		such that every vertex in $V(G)\setminus W$ is adjacent to all but
		$O_F(1)$ vertices of $C$. Let $R_i$ be the collection of  positive-density profile complete
		to $C$.
		Define
		\[
		Z_i:=\{v\in U_i:N_G(v)\cap C=\emptyset\}, \quad  i\in[d].
		\]
		Then the following results hold:
		\begin{enumerate}[label=(\roman*)]
			\item every edge of $G[U_i]$ has both endpoints in $Z_i$;
			
			\item if $R_i\neq \emptyset$, then the graph $G[Z_i]$ is fork-free for every $i\in[d]$.
		\end{enumerate}
	\end{lem}
	\begin{proof}
		(i) Suppose otherwise. Then there exists an edge
		$xy\in E(G[U_i])$
		such that
		$	x\notin Z_i.	$
		Therefore, $x$ has a neighbor in $C$. Let
		$c\in C	$	be such that
		$	cx\in E(G).$
		Since  every vertex in $U_i$ has only
		$O(1)$ missing neighbors in $C$. Therefore, for sufficiently large
		$n$, we can choose two vertices
		$	y_1,y_2\in U_i
		$ such that $cy_1,cy_2\in E(G)$
		and all required edges for the fork can be found.
		
		Together with the edge $xy$, the vertex $c$ creates a copy of the
		fork $B$. Indeed, $c$ plays the role of the center of the fork,
		while $x,y_1,y_2$ provide the three neighbors and $xy$ gives the
		extra edge.
		Now, since	$|C|=a+1,
		$
		we can choose $a$ additional vertices from $C$ which are disjoint
		from the vertices already used in the fork. Using these vertices as
		one endpoint of a matching and using suitable vertices from $U_i$ for the other endpoints, we obtain a copy of
		$
		M_a\cup B.$

		Finally, 
		because different principal classes are almost completely	joined in a cleaned graph.	
		Fixing the vertices already selected for
		$M_a\cup B$. By  Lemma \ref{g:lem-good-internal},  we can  extend this copy by choosing
		vertices from 	$U_1\cup\cdots\cup U_d$ to obtain
		$
		F=(M_a\cup B)\vee K(m_1,\ldots,m_{d-1}),
		$
		contradicting that $G$ is $F$-free.	
		Therefore every internal edge of $G[U_i]$ has both endpoints in
		$Z_i$.

		(ii)	
		If $G[Z_i]$ contained a fork, choose $a$ vertices of $C$ and pair them with
		$a$ fresh vertices of $R_i$ to form $M_a$, and again extend greedily through
		the other principal classes. This also gives $F$, so $G[Z_i]$ is fork-free.	
		
	\end{proof}
	
	After applying Lemma~\ref{g:lem:exact-charging}, each principal class
	contains no vertices that miss the core and hence has the canonical
	join structure. We next analyze the possible decomposition of the
	principal classes into symmetric blocks.
	
	\begin{lem}
		\label{g:lem-beta}
		Let $U_i$ be a cleaned principal class and let
		$Q_1,\ldots,Q_t$ be pairwise disjoint, mutually isomorphic connected
		blocks 
		in $G[U_i]$, where $|V(Q_s)|=O_F(1)$ and
		$t=\Omega_F(n)$.  
		Suppose that all its blocks have the same adjacency pattern to the bounded $W$.  If
		$|\alpha(Q)|\le C_0|V(Q)|$, then
		$|\beta(Q)|\le |\alpha(Q)|.$
	\end{lem}
	
	\begin{proof}
		Let \(Q\) be a fixed block. We construct a new graph \(G_Q\) by
		deleting all edges in \(\alpha(Q)\) and adding all missing pairs in
		\(\beta(Q)\). Hence	
		$e(G_Q)-e(G)=|\beta(Q)|-|\alpha(Q)|.
		$ 	Since \(G\) is extremal, if \(G_Q\) is  \(F\)-free, then
		\(e(G_Q)\le e(G)\) and 
		\(|\beta(Q)|\le|\alpha(Q)|\).
		
		We next prove that \(G_Q\) is still \(F\)-free. Suppose otherwise,
		and let \(F^*\subseteq G_Q\) be a copy of \(F\). Since the only new
		edges of \(G_Q\) are the pairs in \(\beta(Q)\), every edge of
		\(F^*\) which is not an edge of \(G\) must be incident with \(Q\) and
		must join \(Q\) to another principal class. Moreover, since all edges
		inside \(Q\) belonging to \(\alpha(Q)\) were deleted, the induced graph
		\(F^*[Q]\) contains no such edge.
		
		Let \(S\) be the set of vertices of \(F^*-Q\) adjacent in \(F^*\) to
		vertices of \(Q\). Since \(F\) is fixed, we have
		\(|S|\le |V(F)|\). We replace the copy of \(Q\) by another block of the
		same block family.
		For every vertex \(s\in S\setminus W\), the cleaning
		property of \(G\) implies that \(s\) has only \(o(n)\) missings
		in other principal class. Therefore, among the \(t=\Omega_F(n)\) blocks, only
		\(o(n)\) blocks fail to be complete to \(s\). Since \(S\) has bounded
		size, the number of blocks which fail to be complete to some vertex of
		\(S\) is still \(o(n)\). Hence there exists a block \(Q'\) in the same
		family for \(Q\) such that every vertex of \(S\setminus W\) is adjacent to
		all vertices of \(Q'\).
		Since every block has the same
		adjacency pattern to the bounded $W$, \(Q'\) has
		the same adjacency to the vertices of \(S\setminus W\) as \(Q\).
		
		Now map \(F^*\cap Q\) to \(Q'\) through the fixed isomorphism between
		\(Q\) and \(Q'\). The internal edges required inside this part are
		present because the two blocks are isomorphic. The edges between this
		part and \(S\) are also present. Finally, no edge inside \(Q\) is required to be restored,
		because all edges of \(\alpha(Q)\) were removed from \(G_Q\).
		Consequently the obtained copy is also contained in \(G\), yielding
		\(F\subseteq G\), a contradiction. Therefore \(G_Q\) is \(F\)-free, completing the proof.
	\end{proof}

	\subsection{Exact determination of the extremal graph}
	\begin{proof}[Proof of Theorem \ref{g:thm1}]
		Throughout the proof, we use the condition
		\(a+1> \frac{3d}{2}\), which is required in the equality case of
		Lemma~\ref{g:lem-profile}.	Choose a  optimal maximum \(d\)-cut of \(G\), and
		denote the corresponding principal classes by
		\(V_1,\ldots,V_d\). 
		By Theorem~\ref{g:lem-high-bounded}  we further obtain  \(I=\sum_{i\in[d]}e(G[V_i])=O_F(n)\) and 
		\(\sum_{i<j}(|V_i||V_j|-e(V_i,V_j))=O_F(n)\). Hence Lemma~\ref{g:lem-linear-cleaning} gives a bounded	exceptional set \(W\) such that
		$G$ is a cleaned $F$-free graph.

		For every $i\in[d]$, let $U_i=V_i\setminus W$.	 For every \(Y\subseteq W\), let
		\(U_{i,Y}=\{u\in U_i:N_G(u)\cap W=Y\}\). Apply Lemma~\ref{g:lem-profile-reduction} to the partition	\(V(G)=W\cup U_1\cup\cdots\cup U_d\). And
		by Lemma~\ref{g:lem-profile}, we have
		\(e(G)\leq t_d(n)+\frac{a+1}{d}n+o(n)\).
		On the other hand, the extremal construction satisfies
		\(e(G)\ge e(H(n,a+1,d))=t_d(n)+\frac{a+1}{d}n+O_F(1)\). Thus, equality must hold asymptotically in every estimate used in
		Lemma~\ref{g:lem-profile}. Hence its  gives a set
		\(C\subseteq W\) with \(|C|=a+1\) such that every positive-density
		type \(U_{i,Y}\) satisfies \(C\subseteq Y\).
		And
		every vertex
		\(w\in W'=W-C\) is absent from exactly one principal class. We move \(w\) into its unique missing class. 
		This operation does not change the asymptotic number of
		edges, since the number of \(W'\) is
		\(O(1)\). After absorbing all vertices of \(W'\), we may assume that
		the exceptional set is exactly \(C\). The absorption changes only \(O_F(1)\)
		vertices and therefore does not affect any positive-density type.

		For each \(i\in[d]\), define
		\[ Z_i:=\{v\in U_i:N_G(v)\cap C=\emptyset\},
		\quad Q_i:=U_i\setminus Z_i. \]
		Thus \(Q_i\) consists of all vertices in the \(i\)-th principal class
		which have at least one neighbour in  \(C\).
		By Lemma~\ref{g:lem:exact-charging} (i),  all internal edges are contained in
		\(G[Z_i]\) for some \(i\).
		Thus  \(G[Q_i]\) is independent for every \(i\in[d]\), and  then \(Q_i\) has only one positive-density type, namly  $R_i$.	Consider the  blocks	\(\{\{v\}:v\in R_i\}\). For every such block \(Q\), we
		have \(\alpha(Q)=0\). On the other hand, these blocks satisfy the assumptions of
		Lemma~\ref{g:lem-beta}, and therefore 
		\(\beta(Q)=0\). Hence no missing edge exists between two distinct
		sets \(R_i\) and \(R_j\). Consequently, the sets
		\(R_1,\ldots,R_d\) are pairwise complete.

		It remains to consider \(Q_i\setminus R_i\) and \(Z_i\).  
		Since each \(G[Z_i]\) is fork-free by
		Lemma~\ref{g:lem-profile-reduction} (ii), the number of internal edges	satisfies \(I\leq\frac32\sum_{i\in[d]}|Z_i|\). Let \(D\) be the number of
		missing edges compared with the canonical construction
		\(G_{\rm can}=K_{a+1}\join T_d(n-a-1)\). Every vertex in \(Z_i\) is
		anticomplete to the core \(C\), and hence contributes at least
		\(a+1\) missing canonical edges. If \(D_0\) denotes the number of
		missing canonical edges not incident with any \(Z_i\), then
		\(D\geq(a+1)\sum_{i\in[d]}|Z_i|+D_0\).
		Consequently,
		\(e(G)\leq e(G_{\rm can})+I-D\), and hence
		\[
		e(G)\leq
		e(G_{\rm can})-|Q_i\setminus R_i|
		-\left(a-\frac12\right)\sum_{i\in[d]}|Z_i|-D_0 .
		\]
		Because \(a\) is a positive integer,
		\(a-\frac12>0\). Therefore the extremality of \(G\) forces
		\(Q_i\setminus R_i=\emptyset, Z_i=\emptyset\) for every \(i\) and \(D_0=0\).
		
		Thus every vertex  \(V(G)\setminus C\) is complete to the core \(C\), the
		principal classes are pairwise complete, and there are no internal
		edges inside the principal classes. Hence \(G-C\) is a complete
		\(d\)-partite graph. Finally, a complete \(d\)-partite graph with fixed
		order has maximum number of edges only when the sizes of the parts
		differ by at most one. Therefore the principal classes are balanced,
		and we obtain
		\[
		G\cong K_{a+1}\join T_d(n-a-1)=H(n,a+1,d).
		\]
		This completes the proof.
	\end{proof}

	\section{Graph supersaturation}\label{sec:graph-supersaturation}

	\begin{lem}\label{g:lem1}
		Let $F$ be a non-bipartite graph.
		If its extremal graph is unique for all sufficiently large $n$,
		then $F$ is stable.
	\end{lem}
	
	\begin{proof}
		Write $\chi(F)=d+1$.
		The Erd\H{o}s--Simonovits stability theorem~\cite{g:Simonovits3}
		implies that every near-extremal $F$-free graph differs from
		$T_d(n)$ in $o(n^2)$ edges after relabeling.
		The same holds for the unique extremal $F$-free graph.
		Thus these two graphs differ in $o(n^2)$ edges after relabeling.
	\end{proof}

	\begin{thm}
		For every integer $d\ge2$ and every integer $a> 3d/2-1$. 	Let $m_1, \dots, m_{d-1}$ be positive integers such that $m_i \ge a+2$ for each $i \in [d-1]$. Define $F = (M_a \cup B) \vee K(m_1, \dots, m_{d-1})$ where $B$ is a fork. Let  $\theta:=(a+5)/(a+3)$ and
		$\eta:=(a+5)/(a+4)$.
		Then there exists a  constant $\varepsilon>0$  such that, for all sufficiently
		large $n$, 
		
		{\rm (i) } if $1\le q\le \varepsilon n^{\eta}$, then $h_F(n,q) <  c(n,F)$, and 
		
		{\rm (ii)}  if $1\le q\le \varepsilon n^{\theta}$, then $h_F(n,q) < q\cdot c(n,F)$.	
		
		\noindent In particular,
		
		{\rm (iii)} if $ q=o (n^{\eta})$, then $h_F(n,q)=o( c(n,F))$, and 
		
		{\rm (iv)}  if $q=o (n^{\theta})$, then $h_F(n,q)=o (q\cdot c(n,F))$.

	\end{thm}
	
	\begin{proof}Let $F_0 := M_a \cup B$, $M:=K(m_1, \dots, m_{d-1})$ and $S:=\sum_{j=1}^{d-1}m_j$.
		Note that $a+1 > 3d/2$. By Theorem \ref{g:thm1}, the unique extremal graph for $F$ is $H := H(n, a+1, d)$. Let $\{R, V_1, \dots, V_d\}$ be a partition of $V(H)$ such that $|R| = a+1$ and 
		\[
		p \leq |V_i| \leq p+1, \quad \text{where } p = \left\lfloor \frac{n-a-1}{d} \right\rfloor.
		\]
		Without loss of generality, assume $|V_i| = p$ for $i = 1, \dots, d-t$ and $|V_i| = p+1$ for $i = d-t+1, \dots, d$. 
		
		\noindent\textbf {1. Counting one-edge benchmark $c(n,F)$}

		For a part \(V_i\), let \(H_i\) be obtained from \(H\) by adding one
		edge inside \(V_i\).  For any copy $K$ of $F$ in $H_i$, let $F_0(K)$ be the induced subgraph isomorphic to $F_0$.

		\begin{claim}\label{g:clm-F0}
			If the added edge lies in \(V_i\), then every copy \(K\) of \(F\) in
			\(H_i\) satisfies $V(F_0(K)) \subseteq V_i \cup R$.
		\end{claim}
		
		\begin{proof}
			Suppose to the contrary that $V(F_0(K)) \cap V_j \neq \emptyset$ for some $j \neq i$.  
			Since $|R| = a+1 \leq m_\ell$ for every $\ell \in [d]$, each $m_\ell$-part of $K$ contains 
			at least one vertex outside $R$.  In particular, some vertex $v$ of $K$ lies in 
			$V_i \cup V_j$ and belongs to the multipartite part $K(m_1, \dots, m_{d-1})$.  
			By the definition of $F = F_0 \vee K(m_1, \dots, m_{d-1})$, this vertex $v$ must be 
			adjacent to all vertices of $F_0(K)$.  However, $F_0(K)$ has vertices in both $V_i$ and 
			$V_j$, while no vertex in $V_i \cup V_j$ can be adjacent to vertices in both parts 
			(in $H_i$, $V_i$ and $V_j$ are independent sets, and a vertex 
			in $V_i$ is not adjacent to other vertices in $V_i$ except the added edge).  
			This contradiction completes the proof.
		\end{proof}
		
		By Claim \ref{g:clm-F0}, $F_0(K)$ embeds into $V_i \cup R$. We count the copies of \(F\) in \(H_i\). Let $u_0$ be the vertex of degree 3 in $B$. Then $u_0 \in R$. Consider $M_{a+1} = M_a \cup \{u_0 x_1\}$, an $(a+1)$-matching in $V_i \cup R$ consisting of the added edge $uv$ inside $V_i$ and $a$ cross edges between $R$ and $V_i$. We embed $F_0$ as follows:
		\begin{itemize}
			\item Choose $a$ vertices from $R$ (leaving one for $u_0$) and $a$ vertices from $V_i \setminus \{u,v\}$, then match them: $\binom{a+1}{a}\binom{|V_i|-2}{a}a!$.
			\item Choose two additional vertices in $V_i$ as the other neighbors of $u_0$: $\binom{|V_i|-a-2}{2}$.
			\item Select which of the $a+1$ matching edges serves as the fork edge $x_1y$ and orient it: $2(a+1)$.
		\end{itemize}
		Thus the number of $F_0$-embeddings in $V_i \cup R$ is
		\[
		N_i = 2(a+1)^2 \binom{|V_i|-2}{a} a! \binom{|V_i|-a-2}{2}.
		\]
		The remaining parts of $K$ form a copy of $K(m_1, \dots, m_{d-1})$ in the other parts. Let
		\[
		M_i := \sum_{\substack{\pi: [d-1] \to [d]\setminus\{i\} \\ \text{bijection}}} \prod_{j=1}^{d-1} \binom{|V_{\pi(j)}|}{m_j}.
		\]
		Then $\N_F(H_i) = N_i \cdot M_i$.
		
		Let $H^1$ and $H^2$ be obtained from $H$ by adding an edge in $V_1$ or $V_{d-t+1}$, respectively. Now compare $H^1$ and $H^2$. With $|V_1| = p$ and $|V_{d-t+1}| = p+1$,
		\[
		\frac{\N_F(H^2)}{\N_F(H
			^1)} 
		= \frac{p-1}{p-a-1} \cdot \frac{p-a-1}{p-a-3} \cdot \frac{M_2}{M_1}.
		\]
		As $p \to \infty$, expanding $\binom{p}{m}$ and $\binom{p+1}{m}$ yields
		\[
		\frac{M_2}{M_1} = 1 - \frac{S}{(d-1)p} + O_F(p^{-2}).
		\]
		Hence
		\[
		\frac{\N_F(H^2)}{\N_F(H^1)} = 1 + \frac{1}{p}\left(a - 2 - \frac{S}{d-1}\right) + O_F(p^{-2}).
		\]
		Since $m_j \geq a+2$ for all $j$, we have $S \geq (d-1)(a+2)$, and hence
		$a - 2 - \frac{S}{d-1} \leq -4 < 0.$
		Therefore, for sufficiently large $p$, $\N_F(H^1) > \N_F(H^2)$.  Consequently, the minimum number of copies of $F$ in an $n$-vertex graph with $\operatorname{ex}(n,F)+1$ edges is attained by $H^2$, and
		\[
		c(n,F) = \N_F(H^2) = 2(a+1)^2 a! \binom{p-1}{a} \binom{p-a-1}{2} \cdot M_2.
		\]
		As $p \to \infty$,
		$\binom{p-1}{a} \sim \frac{p^a}{a!},  \binom{p-a-1}{2} \sim \frac{p^2}{2}, 
		M_2 = \Theta(p^S),$
		where $S = \sum_{j=1}^{d-1} m_j$.  Hence
		$
		c(n,F) = \Theta(p^{a+2+S}) = \Theta(n^{a+2+S}).$
		
		\noindent\textbf{2. The competing construction $H'$}
		
		Next, we construct an $n$-vertex graph $H'$ with $\operatorname{ex}(n,F) + q$ edges and fewer than \(q\cdot c(n,F)\) copies of \(F\). 
		Let $N := a+1+p$ and $E': = \binom{a+1}{2} + (a+1)p + q$. 
		Replace $H[R \cup V_1]$ by a disjoint union of $k$ components $C_1, \dots, C_k$, where 
		$C_1, \dots, C_{k-1}$ are cliques of sizes $s$ and $s+1$, and $C_k$ is the remaining 
		component (which collects the leftover vertices and need not be a clique), such that the 
		total number of edges is exactly $E'$.  Here
		$
		s = \Theta\!\left(1 + \frac{q}{p}\right),  k = \Theta\!\left(\frac{p}{s}\right),$
		and $|C_i| = O(s)$ for all $i \in [k]$.  All edges between
		\(R\cup V_1\) and \(V_2\cup\cdots\cup V_d\), as well as all edges
		between distinct \(V_i,V_j\) with \(i,j\ge2\), are retained. Note that (1) different $C_i$'s have no edges between them; (2) $C_i$ (for $i \le k-1$) 
		is a clique; (3) $V_j$ ($j \ge 2$) is an independent set; (4) $C_i$ and $V_j$ are completely 
		connected; (5) $V_i$ and $V_j$ ($i,j \ge 2$) are completely connected. We next estimate \(\mathcal N_F(H')\).
		
		\noindent\textbf{3. Counting copies in $H'$}
		
		For any copy $F'$ of $F$ in $H'$, let $F'_0(F')$ be the induced subgraph isomorphic to
		$F_0$.
		Let $t: = |V(F'_0(F') \cap (V_1 \cup R)|$.  Clearly $0 \le t \le 2a+5$ since \(|V(F_0)|=2a+5\).
		We distinguish two cases.
		
		\noindent\textbf{Case 1}. $t = 2a+5$.
		
		In this case $F'_0(F') \subseteq H'[V_1 \cup R]$, i.e.\ $F_0$ lies entirely inside the cliques 
		$C_1, \dots, C_k$.  Since the fork $B$ is connected, it must be contained in a single clique $C_i$, 
		giving $O(k \cdot s^5)$ choices.  The $a$ independent edges of $M_a$ are vertex-disjoint, and each lies 
		within some clique, giving $O(k^a \cdot s^{2a})$ choices.  Hence
		\[
		\mathcal{N}_{F_0}(H'[V_1 \cup R]) = O(k^{a+1} s^{2a+5}) 
		= O\bigl((p/s)^{a+1} s^{2a+5}\bigr) 
		= O(p^{a+1} s^{a+4}).
		\]
		Since $F'_0(F')$ occupies no vertices in $V_2, \dots, V_d$, the multipartite part 
		$K(m_1, \dots, m_{d-1})$ embeds into $V_2, \dots, V_d$ exactly as in $H^1$, contributing 
		$\Theta(p^S)$.  Thus
		\[
		\mathcal{N}_F^{(1)}(H') = O(p^{a+1+S} s^{a+4}).
		\]
		
		\noindent\textbf{Case 2}. $t \le 2a+4$.
		
		Now $F'_0(F')$ has at least one vertex outside $V_1 \cup R$.   The outside vertices lie in 
		$V_2 \cup \cdots \cup V_d$ and suppose that \(F'_0(F')\) uses vertices from \(r\ge1\)  distinct parts among them.
		Since a vertex in $V_j$ is not adjacent to other vertices in the same $V_j$, the parts 
		that contain vertices of $F'_0(F')$ are polluted: their remaining vertices cannot be used for 
		$K(m_1, \dots, m_{d-1})$, because a multipartite vertex must be adjacent to all of $F_0$.
		Hence the multipartite part loses at least $r$ parts.  Let \(T\) be the total size of the multipartite classes that must
		therefore be relocated into \(R\cup V_1\). Since \(m_j\ge a+2\), we have $T \ge r(a+2)$.
		
		All vertices of these missing parts must be pairwise adjacent (different parts of the 
		multipartite must be completely connected) and must also be adjacent to all vertices of 
		$F'_0(F') \cap H'[V_1 \cup R]$.  Since different cliques $C_i$ have no edges between them, all 
		these vertices together with $F'_0(F') \cap H'[V_1 \cup R]$ must lie in the same clique 
		$C_i$.
		
		We now bound $r$ and $T$.  The subgraph $F'_0(F') - (V_1 \cup R)$ is not an independent set: 
		$F_0$ has vertex cover number at least $a+2$, so its outside part must contain at least one 
		edge.  Within $V_2 \cup \cdots \cup V_d$, an edge can only exist between different parts, 
		so $F'_0(F')$ occupies at least $2$ distinct external parts.  Moreover, even if $t \le a+1$, 
		the outside part still spans an edge, forcing $r \ge 2$.  If $t \ge a+2$, we only know 
		$r \ge 1$, but then $T \ge a+2$.  In all subcases one checks that
		\[
		t + T \ge 2a+4.
		\]
		(Indeed, if $t \le a+1$ then $r \ge 2$ gives $T \ge 2a+4$; if $t \ge a+2$ then 
		$T \ge a+2$ and $t+T \ge 2a+4$.)
		
		Now count the embeddings.

		If $t \ge 1$, the $t$ vertices of $F_0(F') \cap H'[V_1 \cup R]$ lie in some $C_i$ 
		($k = \Theta(p/s)$ choices, $O(s^t)$ ways to choose the vertices).  
		If $t = 0$, no clique is selected at this stage.  
		The remaining $2a+5-t$ vertices are chosen outside, giving $O(p^{2a+5-t})$ choices.  
		Thus \[
		\mathcal{N}_{F_0}^{(2)}(H') = 
		\begin{cases}
			O(p^{2a+5}), & t = 0, \\[4pt]
			O(p^{2a+6-t} s^{t-1}), & t \ge 1.
		\end{cases}
		\]
		
		
		The unpolluted external parts contribute $\Theta(p^{S-T})$.  
		The missing $T$ vertices must lie in $C_i$.  
		If $t \ge 1$, the clique $C_i$ is already fixed, so these vertices are chosen within $C_i$ 
		in $O(s^T)$ ways.  
		If $t = 0$, no clique is fixed yet; the first vertex has $\Theta(p)$ choices (any vertex 
		in $V_1 \cup R$), fixing the clique, and the remaining $T-1$ vertices are chosen in that 
		clique, giving $O(p \cdot s^{T-1})$ ways.  
		Thus 
		\[
		\N_M^{(2)}(H') = 
		\begin{cases}
			O(p^{S-T+1} s^{T-1}), & t = 0, \\[4pt]
			O(p^{S-T} s^T), & t \ge 1.
		\end{cases}
		\]
		Consequently,
		\[
		\mathcal N_F^{(2)}(H')
		=
		\begin{cases}
			O_F(p^{2a+6+S-T}s^{T-1}),
			&t=0,\\[2mm]
			O_F(p^{2a+6+S-t-T}s^{t+T-1}),
			&t\ge1.
		\end{cases}
		\]
		
		In both subcases, compare with Case~1:
		\begin{align*}
			\frac{\mathcal{N}_F^{(2)}(H')}{\mathcal{N}_F^{(1)}(H')}
			&= O\left(\frac{p^{2a+6+S-t-T} s^{t+T-1}}{p^{a+1+S} s^{a+4}}\right) 
			= O\left(p^{a+5-t-T} s^{t+T-a-5}\right) 
			= O\left(\left(\frac{s}{p}\right)^{t+T-a-5}\right).
		\end{align*}
		(The same formula holds for $t=0$ by taking $p^{2a+6+S-T} \cdot s^{T-1}$ and comparing; 
		the algebra yields $p^{a+5-T} s^{T-a-5}$, and with $T \ge 2a+4$ the exponent of $p$ is 
		negative.)
		Since $t+T \ge 2a+4$, we have $t+T-a-5 \ge a-1 \ge 0$ (as $a \ge 3$).  
		Hence $s/p \le 1$ implies the ratio is $o(1)$. Thus, $\mathcal N_F(H') = O(p^{a+1+S}s^{a+4})$.
		
		\noindent\textbf{4. Compare $q\cdot c(n,F)$ and $\N(H')$}
		
		Since
		$c(n,F)=\Theta_F(p^{a+2+S}),$
		we obtain
		\[\frac{\mathcal N_F(H')}{c(n,F)}
		=
		O_F\!\left(
		\frac1{p}
		+
		\frac{q^{a+4}}{p^{a+5}}
		\right),\quad
		\frac{\mathcal N_F(H')}{qc(n,F)}
		=
		O_F\!\left(
		\frac1{pq}
		+
		\frac{q^{a+3}}{p^{a+5}}
		\right).
		\]
		Recall that $p = \left\lfloor \frac{n-a-1}{d} \right\rfloor$ and \(n\) is sufficiently
		large. If $q\le\varepsilon n^{\frac{a+5}{a+4}},$ we obtain $
		\frac{q^{a+4}}{p^{a+5}}
		\le O(\varepsilon^{a+4}).$ Choosing \(\varepsilon>0\) sufficiently small  gives
		$h_F(n,q)
		\le
		\mathcal N_F(H')<c(n,F)$.
		Moreover, if $
		q\le\varepsilon n^{\frac{a+5}{a+3}},$
		we have $
		\frac{q^{a+3}}{p^{a+5}}
		\le O(\varepsilon^{a+3}).$ Choosing \(\varepsilon>0\) sufficiently small gives
		$h_F(n,q)
		\le
		\mathcal N_F(H')<qc(n,F),$ as required.
		This completes the proof. 
	\end{proof}


	\section{Structural tools for hypergraphs}\label{h:sec:tools}
	This section collects the tools used to prove Theorem~\ref{h:thm:connected-exact}.
	We begin with a matching bound and an embedding lemma based on fat pairs.
	We then state Gerbner's structural theorem and the stability theorem of
	Ma and Qiu.
	These results give a nearly balanced partition for an extremal
	$\mathcal F$-free hypergraph.
	
	The matching number $\nu(H)$ is the largest number of pairwise
	disjoint edges.
	\begin{lemma}\label{h:lem:small-cover}
		Fix integers $r\ge2$ and $s\ge0$. If an $r$-graph $H$ on at most
		$n$ vertices has $\nu(H)\le s$, then there is a set $S$ of at most
		$s$ vertices such that
		$
		e(H-S)=O_{r,s}(n^{r-2}).
		$
		In particular,
		\begin{equation}\label{h:eq:matching-bound}
		e(H)\le s\binom n{r-1}+O_{r,s}(n^{r-2}).
		\end{equation}
	\end{lemma}
	
	\begin{proof}
		There is nothing to prove when $s=0$. Otherwise, let $U$ be the
		union of the edges of a maximal matching. Then $|U|\le rs$, and
		every edge meets $U$. For $x\in U$, let
		$
		H_x=\{e\in E(H):e\cap U=\{x\}\}.
		$
		Let $S$ contain those $x\in U$ for which
		$|H_x|>rs\binom n{r-2}$. We claim that $|S|\le s$.
		If $x_1,\ldots,x_{s+1}\in S$ were distinct, we could choose
		disjoint edges $e_i\in H_{x_i}$ in order. Before choosing $e_i$,
		at most $s(r-1)$ vertices outside $U$ have been used. Each such
		vertex lies in at most $\binom n{r-2}$ members of $H_{x_i}$.
		No other vertex of $U$ lies in a member of $H_{x_i}$.
		An edge is therefore available at each step. This contradicts
		$\nu(H)\le s$.
		
		An edge avoiding $S$ either meets $U$ once or at least twice.
		The first type has at most $|U|rs\binom n{r-2}$ members.
		The second has at most $\binom{|U|}{2}\binom n{r-2}$ members.
		This proves $
		e(H-S)=O_{r,s}(n^{r-2})
		$. Counting the edges that meet $S$
		proves \eqref{h:eq:matching-bound}.
	\end{proof}

	
	For a graph $G$,  $G^{+k}$ replaces each
	edge $xy$ by $\{x,y\}\cup Z_{xy}$.
	The vertices  from $V(G)$ are its \emph{core vertices}.
	The vertices in the sets $Z_{xy}$ are its \emph{private vertices}.
	Fix a graph $G$, and let  $f=v(G^{+k})=v(G)+(k-2)e(G)$. A pair $xy$ is
	\emph{fat in $H$} if there are $f$ edges of $H$ whose pairwise
	intersection is exactly $\{x,y\}$.  Let $\mathsf F(H)$ be the 
	graph formed by the fat pairs.
	
	\begin{lemma}\label{h:lem:fat}With this notation,
		the following statements hold.
		\begin{enumerate}[label=(\roman*)]
			\item A pair that is not fat lies in $O_{G^{+k}}(n^{k-3})$ edges of
			$H$. There are $O_{G^{+k}}(n^{k-1})$ edges containing such a pair.
			For a fixed vertex $w$, there are $O_{G^{+k}}(n^{k-2})$ edges
			containing a nonfat pair $wx$.
			\item Suppose that some edges of a copy of $G$ have already been
			expanded. If every remaining core edge is fat and
			the unused core vertices avoid all private vertices already
			chosen, then the partial expansion extends to a copy of $G^{+k}$.
			\item If $H$ is $G^{+k}$-free, then $\mathsf F(H)$ is $G$-free.
		\end{enumerate}
	\end{lemma}
	
	\begin{proof}
		For a nonfat pair $xy$, the family
		$\{e\setminus\{x,y\}:xy\subseteq e\in E(H)\}$ has no matching
		of size $f$. The vertices of a maximal matching form a vertex cover of size
		at most $(k-2)(f-1)$. Therefore its edges number at most
		$(k-2)(f-1)\binom n{k-3}$. This also holds for $k=3$.
		Sum over all pairs, or only pairs containing $w$, to prove (i).
		
		For (ii), take the remaining core edges in order. For a fat pair
		$xy$, its $f$  edges in the definition of fat are disjoint outside $xy$.
		The other core vertices and all private vertices already chosen
		form a forbidden set of size less than $f$. Each forbidden vertex
		meets at most one fat edge outside $xy$. Thus at least one fat
		edge avoids this set. Repeat this argument until all edges are expanded.
		Part (iii) is a special case and follows by starting with no expanded edges.
	\end{proof}
	
	We use the following part of Gerbner's structure theorem.
	It does not require the  graph $Q$ to be connected. We state only the parts needed here.
	
	\begin{theorem}[Gerbner {\cite[Theorem 2.5]{h:Gerbner2025}}]
		\label{h:thm:structure}
		Let $Q$ be an  ordinary graph with $\chi(Q)=r+1>k$, and let $H$ be an extremal
		$Q^{+k}$-free $k$-graph on $n$ vertices. For sufficiently large $n$,
		there is a partition $V(H)=V_1\sq\cdots\sq V_r$ and a set
		$X$ of bounded size such that, in the graph $\mathsf F(H)$, with fatness defined using $f=v(Q^{+k})$,
		\begin{enumerate}[label=\textup(\roman*)]
			\item every vertex of $X$ has $\Omega(n)$ neighbors in each part $V_i$;
			\item every vertex of $V_i\setminus X$ has $o(n)$ neighbors in
			$V_i$ and $|V_j|-o(n)$ neighbors  in each $V_j$,
			$j\ne i$.
		\end{enumerate}
		The estimates are uniform over the stated vertices.
	\end{theorem}

	Let $\#K_s(A)$ count the $s$-element
	vertex sets that form a clique in $A$.  Here each clique vertex set is counted once. 
	
	\begin{lemma}\label{h:lem:balance}
		Suppose that $\chi(G)=k+1$, and let $H$ be an extremal
		$G^{+k}$-free $k$-graph with $e(H)\ge t^{(k)}(n)-O(n^{k-1})$.
		The partition in Theorem~\ref{h:thm:structure} satisfies
		$|V_i|=n/k+o(n)$ for every $i$.
	\end{lemma}
	\begin{proof}
		By Lemma~\ref{h:lem:fat}(i), 
		$\#K_k(\mathsf F(H))\ge e(H)-O(n^{k-1})\ge t^{(k)}(n)-O(n^{k-1}). $
		Theorem \ref{h:thm:structure} gives only $o(n^2)$ fat pairs within the
		parts, including those incident with the bounded exceptional set.
		Thus only $o(n^k)$ of these cliques are noncrossing, and
		$\prod_{i\in[k]}|V_i|\ge \#K_k(\mathsf F(H))-o(n^k)\ge  t^{(k)}(n)-o(n^k)$.
		The ratios $|V_i|/n$ have sum one. Their product is at most
		$k^{-k}$, with equality only when all ratios are $1/k$.
		If balance failed, a convergent subsequence would give unequal
		limiting ratios with product at least $k^{-k}$, a contradiction.
	\end{proof}
	
	The next theorem will  be used for stability.  
	\begin{theorem}[Ma--Qiu {\cite[Theorem 1.4]{h:MaQiu2020}}]
		\label{h:thm:stability-input}
		Let $\chi(Q)=k+1>s\ge2$. If a $Q$-free graph $P$ on $n$
		vertices has at least
		$
		\#K_s(T_k(n))-o(n^s)
		$
		copies of $K_s$, then $P$ differs from the
		balanced complete $k$-partite graph $T_k(n)$ in
		$o(n^2)$ edges  after relabeling.
	\end{theorem}
	\begin{proposition}\label{h:prop:stability}
		Let $\chi(G)=k+1$, and suppose that $\J_a(n)$ is the unique
		extremal $G^{+k}$-free $k$-graph for all large $n$, where $a$ is
		fixed. Then $G^{+k}$ is stable.
	\end{proposition}
	\begin{proof}
		Let $H_n$ be a $G^{+k}$-free $k$-graph with
		$e(H_n)=\ex(n,G^{+k})-o(n^k)=t^{(k)}(n)-o(n^k)$.
		Lemma~\ref{h:lem:fat} shows that $\mathsf F(H_n)$ is $G$-free
		and has at least $t^{(k)}(n)-o(n^k)$ cliques of order $k$.
		Theorem~\ref{h:thm:stability-input} places it within $o(n^2)$
		edge changes of a balanced complete $k$-partite graph.
		Each  edge belongs to at most $\binom n{k-2}$ sets of
		$k$ vertices. Hence the clique hypergraph of $\mathsf F(H_n)$
		differs from $\T^{(k)}(n)$ in $o(n^k)$ edges.
		All but $O(n^{k-1})$ edges of $H_n$ lie in this clique hypergraph.
		Therefore $H_n$ has only $o(n^k)$ edges outside $\T^{(k)}(n)$.
		Its edge count also bounds the missing crossing edges by $o(n^k)$.
		Finally, after relabeling, $\T^{(k)}(n)$ differs from $\J_a(n)$
		in $O(n^{k-1})$ edges. Choose the $a$ distinguished vertices so
		that the remaining parts are balanced. Only edges meeting these
		vertices can change. This proves stability.
	\end{proof}
	
	\section{The exact extremal theorem for hypergraphs}\label{h:sec:connected}
	
	Throughout this section, assume \eqref{h:eq:admissible} and use
	$G,G^{+k}$ from \eqref{h:eq:connected-parameters}. Let $f=v(G^{+k})$.
	The graph $G$ is connected and has chromatic number $k+1$.
	Its expansion is therefore connected and linear.
	We prove Theorem~\ref{h:thm:connected-exact}.
	
	\subsection{The candidate and local uniqueness}
	We  use $J(S;V_1,\ldots,V_k)$ for the same construction
	with specified  parts $V_1,\ldots,V_k$, which need not be balanced.
	\begin{lemma}\label{h:con:lem:jfree}
		Every hypergraph $J(S;V_1,\ldots,V_k)$ with $|S|=a$ is $G^{+k}$-free.
		Moreover, $G^{+k}$ is not $k$-partite.
	\end{lemma}
	\begin{proof}
		In $M_a\sq K_{1,b}$, take all $a$ matching edges and one edge of
		the star. These $a+1$ edges are pairwise vertex-disjoint.
		Next, complete each edge to a $K_{k+1}$.
		For every such edge, choose one vertex from each of the other $k-1$ parts.
		Since every $m_j\ge a+1$, we can use different vertices for different copies.
		Thus $G$ contains $a+1$ vertex-disjoint copies of $K_{k+1}$.
		Their expansions in $G^{+k}$ are also vertex-disjoint.
		
		A $k$-partite $k$-graph cannot contain $K_{k+1}^{+k}$.
		Indeed, its $k+1$ core vertices would have to lie in different parts.
		This also shows that $G^{+k}$ is not $k$-partite.
		Suppose that $G^{+k}$ embeds into $J=J(S;V_1,\ldots,V_k)$.
		Since $J-S$ is $k$-partite, each expanded clique must meet $S$.
		The $a+1$ cliques are disjoint, which contradicts $|S|=a$.
		
	\end{proof}
	
	\begin{lemma}\label{h:con:lem:local}
		Fix $\eta>0$ and let $J=J(S;V_1,\ldots,V_k)$.
		Assume that $|S|=a$ and $|V_i|\ge\eta n$ for every $i$.
		There is a constant $c_1>0$ for which the following results hold
		for all sufficiently large $n$.
		This constant depends only on $G^{+k}$ and $\eta$.
		\begin{enumerate}[label=(\roman*)]
			\item For every $e\notin E(J)$,
			$ N_{G^{+k}}(J+e)=\Theta(n^{f-a-k}).
			$
			\item If a $G^{+k}$-free $k$-graph $H$ on the same vertex set satisfies
			$|E(J)\setminus E(H)|\le c_1n^{k-1}$, then $H\subseteq J$.
		\end{enumerate}
	\end{lemma}
	\begin{proof}
		We first give a lower bound for $ N_{G^{+k}}(J+e).
		$
		A non-edge $e$ avoids $S$ and contains two vertices $x,y$ in some
		part $V_i$. Fix one edge of $M_a$.
		Map its core endpoints to $x,y$ and its private vertices to
		$e\setminus\{x,y\}$.
		Map the star center and one endpoint of each remaining matching edge to $S$.
		Use each vertex of $S$ exactly once.
		Choose all remaining forest core vertices in $V_i$. The $h$ extra independent vertices are also placed in $V_i$.
		Choose the other $k-1$ core classes in the other parts, one class per part.
		Their sizes are $m_1,\ldots,m_{k-1}$ in any fixed order.

		Every core edge that avoids $S$ and is not the first matching edge
		has its endpoints in distinct  parts.
		Choose one private vertex from each of the other $k-2$ parts to
		make its expansion crossing.
		For a core edge meeting $S$, choose any new private vertices.
		This is possible because every $k$-set meeting $S$ is an edge of $J$.
		Make these choices one at a time, and  avoid vertices already used.
		Exactly $a+k$ vertices have been fixed, and each other vertex has at least $\eta n-O_{G^{+k}}(1)$ choices.
		Thus we obtain at least
		$c_0n^{f-a-k}$ embeddings in $J+e$.
		
		Next, we give an
		upper bound  for $ N_{G^{+k}}(J+e)$.
		Fix the $a+1$ disjoint expanded cliques from
		Lemma~\ref{h:con:lem:jfree}.
		In any embedding of $G^{+k}$ into $J+e$, some edge of one of these
		expanded cliques must be mapped to $e$.
		Otherwise all $a+1$ cliques would lie in $J$.
		Each would have to meet $S$, which is impossible.
		The other $a$ expanded cliques each meet $S$. Since they are
		vertex-disjoint, the embedding uses every vertex of $S$.
		It also uses every vertex of $e$.
		Thus it contains the fixed set $S\cup e$, which has size $a+k$.
		There are only a constant number of ways to embed  these vertices.
		The remaining vertices then have at most $n^{f-a-k}$ choices in total.
		This proves the upper bound, and then
		proves (i).
		
		Suppose that \(H\not\subseteq J\). Choose \(e\in E(H)\setminus E(J)\). Consider the embeddings constructed for the lower bound in (i).
		These embeddings have one more property.
		Every edge other than $e$ contains at most one fixed vertex from $S\cup e$.
		Indeed, the forest core vertices assigned to $S$ are pairwise nonadjacent in $G$.
		They are also not adjacent to either endpoint of the first matching edge.
		The private vertices assigned to $e$ occur
		in no other edge.
		Thus a fixed missing edge of $J$ destroys at most
		$ C_1n^{f-a-k-(k-1)}
		$
		of the constructed embeddings, where $C_1$ depends only on $G^{+k}$.
		To see this, choose the edge of $G^{+k}$ mapped to the missing edge.
		Then choose the image of each of its vertices.
		There are only a constant number of choices to embed these vertices.
		
		Choose $c_1<c_0/(2C_1)$.
		At most $c_1n^{k-1}$ edges of $J$ are missing from $H$.
		These missing edges cannot destroy all the embeddings above.
		Hence $H$ contains a copy of $G^{+k}$, a contradiction.
		This proves (ii).

	\end{proof}

	For any fixed integer $a\ge0$, balancing the part sizes gives
	\begin{align}\label{h:eq:tk-shift}
	t^{(k)}(n-a)=t^{(k)}(n)-a(n/k)^{k-1}+O(n^{k-2}).
	\end{align}
	By the definition of $\Ja(n)$ and Lemma~\ref{h:con:lem:jfree},
	\begin{align}\label{h:con:eq:j-expansion}
	\ex(n,G^{+k})\ge& e(\Ja(n))=\binom nk-\binom{n-a}k+t^{(k)}(n-a)\notag\\
	=&t^{(k)}(n)+a\alpha (n/k)^{k-1}+O_{G^{+k}}(n^{k-2}).
	\end{align}

	\subsection{Non-crossing edges outside a bounded set}
	Let $H$ be an extremal $G^{+k}$-free $k$-graph.
	Since $\chi(G)=2+(k-1)=k+1$, we apply Theorem~\ref{h:thm:structure} and
	fix a resulting partition $V_1,\ldots,V_k$ and
	exceptional set $X$.
	We call the vertices outside $X$ \emph{ordinary}.
	Lemma~\ref{h:lem:balance} gives $|V_i|=n/k+o(n)$ for every $i$.
	
	The next lemma starts with images of the vertices of
	$M_a\sq K_{1,b}$, which is a subgraph of the core graph $G$.
	It then embeds the remaining vertices and edges of $G^{+k}$.
	
	\begin{lemma}\label{h:con:lem:completion}
		Let $L=M_a\sq K_{1,b}$.
		Fix an injective map $\varphi:V(L)\to V(H)$.
		Suppose that the following conditions hold for some $i\in[k]$.
		\begin{enumerate}[label=(\roman*)]
			\item 	Let $R=\varphi(V(L))\cap X$ be the set of images in $X$. Every image outside $X$ lies in $V_i\setminus X$.
			
			\item Some edges $uv$ of $L$ have  expansions
			$
			\{\varphi(u),\varphi(v)\}\cup Y_{uv}\in E(H)$ with $|Y_{uv}|=k-2$.
			These sets $Y_{uv}$ are pairwise disjoint and avoid $\varphi(V(L))$.
			Every remaining edge $uv$ of $L$ satisfies
			$\varphi(u)\varphi(v)\in E(\mathsf F(H))$.
			\item In each $V_j\setminus X$ with $j\ne i$, the vertices of $R$
			have $\Omega(n)$ common neighbors in $\mathsf F(H)$.
		\end{enumerate}
		Then $\varphi$ extends to an embedding of $G^{+k}$ in $H$.
	\end{lemma}
	
	\begin{proof}
		The images of $V(L)$ have already been chosen.
		The sets $Y_{uv}$ will be the images of the corresponding private vertices.
		We choose every further image to avoid all vertices already used.

		First choose   $h$ vertices  in $V_i\setminus X$.
		There are $n/k+o(n)$ available vertices before excluding the bounded
		set already used. Note that no edge of $G$ joins $I_{h}$ to $V(L)$.
		Assign the other $k-1$ classes of $G$ bijectively to the parts
		$V_j$ with $j\ne i$. Choose their vertices one at a time.
		Start with the common
		$\mathsf F(H)$-neighborhood of $R$, which  has $\Omega(n)$ vertices.
		Require each new
		vertex to be adjacent in $\mathsf F(H)$ to every previously chosen ordinary vertex
		in a different part. 
		By Theorem~\ref{h:thm:structure}, each such image excludes only $o(n)$ choices.
		There are only boundedly many previously chosen vertices.
		Excluding them and all  private-vertex images still leaves
		$\Omega(n)$ choices. No adjacency is required within one class of $G$.
		Continue until every vertex of $G$ has an image.

		Every unexpanded edge of $L$ maps to a fat pair by (ii).
		Every edge of $G$ outside $L$ joins two different classes of the join.
		If one endpoint maps to $R$, the required pair is fat by (iii)
		and our choice of images. Otherwise, both endpoints have ordinary
		images in different parts. The pair is also fat as above.
		Thus every edge still to be expanded maps to an edge of $\mathsf F(H)$.

		Apply Lemma~\ref{h:lem:fat}(ii).
		At each step, avoid the images of all other vertices of $G$
		and all private vertices already used.
		This forbidden set has fewer than $f$ vertices.
		The $f$ fat edges for a fat pair are disjoint outside that pair.
		Hence one fat edge avoids the forbidden set.
		Repeating this choice completes an embedding of $G^{+k}$ in $H$.
	\end{proof}

	For each $i\in[k]$, let
	\[
	\C_i=\{e\in E(H):e\cap X=\varnothing,\ |e\cap V_i|\ge2\}.
	\]
	Let $S_{b+1}$ be a star consisting of $b$ edges with a common
	center, and  consider $S^{+k}_{b+1}$.
	Suppose its center is in $V_i$ and all its edges lie in $\C_i$.
	Each edge then contains another vertex of $V_i$.
	We can use that vertex as its core leaf.
	
	\begin{lemma}\label{h:con:lem:bounded-W}
		There is a set $W\supseteq X$ with $|W|=O_{G^{+k}}(1)$ such that:
		\begin{enumerate}[label=(\roman*)]
			\item if $X\ne\varnothing$, every edge of $H-W$ is crossing;
			\item if $X=\varnothing$, the number $\C$  of non-crossing edges
			in $H-W$ satisfies
			\begin{equation}\label{h:con:eq:residual}
			\C\le\frac{b-1}{2}\,n\binom n{k-2}+O_{G^{+k}}(n^{k-2}).
			\end{equation}
		\end{enumerate}
	\end{lemma}
	\begin{proof}
		First suppose that $X\ne\varnothing$.
		Suppose some $\C_i$ contains $a$ pairwise disjoint edges.
		Choose two $V_i$-vertices from each edge as its matching core endpoints.
		Use its other vertices as private vertices.
		Take $x\in X$ as the star center and choose $b$ 
		$\mathsf F(H)$-neighbors of $x$ in $V_i\setminus X$ as its leaves.
		These leaves can be chosen to avoid the matching edges.
		The vertex $x$ also has $\Omega(n)$ neighbors in every other part by Theorem \ref{h:thm:structure}.
		Lemma~\ref{h:con:lem:completion} gives a copy of $G^{+k}$, a contradiction.
		Hence $\nu(\C_i)\le a-1$.
		For each $i$, take a maximal matching in $\C_i$.
		Let $W$ be the union of $X$ and the vertices of these matchings.
		Then $|W|\le|X|+k^2(a-1)$.
		Every non-crossing edge meets $W$, which proves (i).
		
		Now suppose that $X=\varnothing$.
		Suppose $\C_i$ contains an $S^{+k}_{b+1}$ centered in $V_i$.
		The star uses $1+b(k-1)$ vertices.
		If $\C_i$ had a matching of size at least $a+b(k-1)+1$,
		we could delete the matching edges that meet the star.
		At least $a$ matching edges would remain.
		Together with the star, these give an embedding of $L^{+k}$
		with all core vertices in $V_i$.
		Lemma~\ref{h:con:lem:completion} again gives a contradiction.
		Thus, for every such $\C_i$,
		$\nu(\C_i)\le a+b(k-1)$.
		Let $W$ be the union of the vertices of maximal matchings in
		all these families. Then
		$|W|\le k^2(a+b(k-1))$.
		
		Write $\C'_j=\{e\in\C_j:e\cap W=\varnothing\}$.
		If $\C_j$ contained a linear $b$-star centered in $V_j$,
		its maximal matching was included in the construction of $W$,
		so $\C'_j$ is empty. Thus every  $\C'_i$ contains no such star.
		For any $x\in V_i\setminus W$, the  $(k-1)$-graph
		$
		\{e\setminus\{x\}:x\in e\in\C'_i\}
		$
		has  matching number at most $b-1$.
		Apply \eqref{h:eq:matching-bound},  with $r=k-1$ and $s=b-1$.
		It gives the uniform bound
		\[
		d_{\C'_i}(x)\le(b-1)\binom n{k-2}+O_{G^{+k}}(n^{k-3}).
		\]
		Each edge of $\C'_i$ contains at least two vertices of $V_i$.
		Summing degrees over $V_i\setminus W$ therefore gives
		\[
		e(\C'_i)\le\frac{b-1}{2}\,|V_i|\binom n{k-2}
		+O_{G^{+k}}(n^{k-2}).
		\]
		Every non-crossing edge of $H-W$ belongs to at least one
		of these families. Summing over $i$ and using
		$\sum_i|V_i|=n$ proves \eqref{h:con:eq:residual}.
	\end{proof}

	Suppose that for arbitrarily large $n$, there is an extremal
	$G^{+k}$-free $k$-graph that is not isomorphic to $\Ja(n)$.
	Choose a sequence of these hypergraphs with $n\to\infty$.
	The candidate might have fewer than $\ex(n,G^{+k})$ edges, or another extremal
	hypergraph might have the same number of edges.
	
	Lemma~\ref{h:con:lem:bounded-W} bounds $|X|$ and $|W|$ by constants.
	Pass to a subsequence on which these sizes are fixed.
	Label the vertices of $W$.
	Pass to a further subsequence so that each label has fixed membership in $X$
	and lies in a fixed part.
	We use the same labels for the corresponding vertices at different orders.
	Let $X_i=V_i\setminus W$.
	By Lemma~\ref{h:lem:balance}, each $X_i$ has size $n/k+o(n)$.
	
	For $w\in W$, let
	\begin{align*}
	\mathcal L_w=\{e\in E(H):e\cap W=\{w\}\},
	\quad\mathcal L_{wi}&=\{e\in\mathcal L_w:|e\cap X_i|\ge2\}. 
	\end{align*}
	There are only finitely many pairs $(w,i)$.
	The ratios ${|\mathcal L_{wi}|\over (n/k)^{k-1}}$ are bounded.
	We may therefore pass to a further subsequence on which they all converge.
	Define
	\begin{equation*}
	W_i=\left\{w\in W:
	\lim_{n\to\infty}\frac{|\mathcal L_{wi}|}{{(n/k)}^{k-1}}>0\right\}.
	\end{equation*}
	Thus $|\mathcal L_{wi}|=\Omega(n^{k-1})$ when $w\in W_i$,
	and $|\mathcal L_{wi}|=o(n^{k-1})$ otherwise.
	The constant in the first bound may depend on the positive limit.
	It is fixed along this subsequence.
	All later asymptotic statements in this section refer to the chosen subsequence.
	
	We record one consequence of the fat-pair estimate.
	For a fixed  $w\in W$,   only $O_{G^{+k}}(n^{k-2})$ edges of $\mathcal L_w$
	contain a non-fat pair $wx$ with   $x\in V(H)\setminus W$.
	If $w\in V_j\setminus X$, then $w$ has only $o(n)$ fat neighbors in $X_j$ by Theorem \ref{h:thm:structure}.
	Edges using one of these neighbors number at most $o(n)\binom n{k-2}$.
	It follows that
	\begin{equation}\label{h:con:eq:ordinary-own-part}
	|\{e\in\mathcal L_w:e\cap X_j\ne\varnothing\}|
	=o(n^{k-1})
	\qquad(w\in W\cap(V_j\setminus X)).
	\end{equation}
	In particular, such a vertex $w$ does not belong to $W_j$.
	
	The next counting lemma will be used twice.
	The part type of a set records how many of its vertices lie in each part.
	
	\begin{lemma}\label{h:con:lem:patterns}
		Let $\ell\in\{k-1,k\}$.
		Take $\ell$ parts, each of size $n/k+o(n)$.
		Choose $r$ of them as the only parts allowed to contribute two or more vertices to a $k-1$-set.
		The number of such $k-1$-sets that use a part more than once
		is at most
		\[
		\left(\frac r\ell Q_\ell+o(1)\right) (n/k)^{k-1}, \qquad Q_\ell=\frac{\ell^{k-1}}{(k-1)!}-\binom\ell {k-1}.
		\]
	\end{lemma}
	\begin{proof}
		A type is a vector $(u_1,\ldots,u_\ell)$ of nonnegative
		integers summing to $k-1$.
		The number of sets of this type,
		divided by $(n/k)^{k-1}$, tends
		to $1/\prod_{i\in[\ell]} u_i!$.
		Group the types by their nonempty repeated-part set
		$\{i:u_i\ge2\}$.
		By symmetry, the total coefficient for a fixed repeated-part set depends only on its size $j$.
		Call this coefficient $c_j\ge0$.
		The coefficient for the allowed sets is therefore
		$\sum_{j\ge1}c_j\binom rj$.
		If every part is allowed to repeat, the coefficient is
		\[
		\sum_{j\ge1}c_j\binom\ell j
		=\frac{\ell^{k-1}}{(k-1)!}-\binom\ell {k-1}=Q_{\ell}.
		\]
		For every $1\le j\le\ell$, we have $\binom rj/\binom\ell j\le r/\ell$.
		The numerator is zero when $j>r$.
		Comparing terms proves the bound.
		The case $r=0$ is immediate.
	\end{proof}
	
	\subsection{Determining the exceptional set}
	\begin{lemma}\label{h:con:lem:B-nonempty}
		Assume \eqref{h:eq:admissible}.
		In the sequence chosen above, $X=\varnothing$ cannot hold along an infinite subsequence.
	\end{lemma}
	\begin{proof}
		Suppose that $X=\varnothing$ along an infinite subsequence. Restrict to this subsequence.
		Take the set $W$ from Lemma~\ref{h:con:lem:bounded-W}(ii). Let
		$
		I=\{i:|W_i|\ge a+b\}.
		$
		Fix $i\in I$. 
		We first claim that for each $x\in X_i$, at most $b-1$
		vertices $w\in W_i$ satisfy
		\begin{equation}\label{h:con:eq:high-degree}
		d_i(w,x):=|\{e\in\mathcal L_{wi}:x\in e\}|>f n^{k-3}.
		\end{equation}
		
		If not, choose $b$ such vertices $w_1,\ldots,w_b$.
		Choose edges $e_j\in\mathcal L_{w_ji}$ containing $x$, one at a time.
		Require them to be pairwise disjoint outside $x$.
		Each extra forbidden vertex lies in at most $\binom n{k-3}\le n^{k-3}$ edges containing $w_j,x$.
		Fewer than $f$ vertices need to be avoided.
		Hence \eqref{h:con:eq:high-degree} leaves at least one choice.
		The condition $e_j\cap W=\{w_j\}$ already excludes all other vertices of $W$.
		Choose one vertex of $(e_j\cap X_i)\setminus\{x\}$ as the
		core leaf of each edge.
		This gives a $S^{+k}_{b+1}$ centered at $x$.
		Each $w_j$ is a private vertex.
		
		Choose another $a$ distinct vertices from
		$W_i\setminus\{w_1,\ldots,w_b\}$.
		For each chosen $w$, the family $\mathcal L_{wi}$ has
		$\Omega(n^{k-1})$ edges.
		Avoiding a fixed additional vertex excludes only
		$O(n^{k-2})$ of them.
		We can therefore choose $a$ pairwise disjoint edges
		avoiding the star. In each, choose two $X_i$-vertices
		as its matching core endpoints.
		We have obtained an embedding of $(M_a\cup K_{1,b})^{+k}$ with all core
		vertices in $X_i$.
		Lemma~\ref{h:con:lem:completion} gives a copy of $G^{+k}$,
		a contradiction. This proves the claim.
		
		Fix $x\in X_i$.
		At most $b-1$ vertices of $W_i$ satisfy \eqref{h:con:eq:high-degree}.
		Bound the corresponding values of $d_i(w,x)$ by $\binom n{k-2}$.
		Bound every other term by $f n^{k-3}$.
		Each edge of $\mathcal L_{wi}$ is counted at least twice
		when we sum over $x\in X_i$.
		Since $|W|$ is bounded, it follows that
		\begin{equation}\label{h:con:eq:large-Pi}
		\sum_{w\in W_i}|\mathcal L_{wi}|
		\le\frac{b-1}{2}|X_i|\binom n{k-2}
		+O_{G^{+k}}(n^{k-2}).
		\end{equation}

		For each $w$, delete every edge of $\mathcal L_w$ that uses some $X_i$, $i\in I$, at least twice.
		Call the remaining family $\mathcal L'_w$.
		If $w\notin W_i$, then $|\mathcal L_{wi}|=o(n^{k-1})$.
		There are only boundedly many such pairs $(w,i)$.
		Summing \eqref{h:con:eq:large-Pi} over $i\in I$ and adding
		\eqref{h:con:eq:residual}, we obtain
		\begin{align}
		\C+\sum_{w\in W}|\mathcal L_w\setminus\mathcal L'_w|
		&\le\frac{b-1}{2}
		\left(n+\sum_{i\in I}|X_i|\right)\binom n{k-2}+o(n^{k-1})\nonumber\\
		&\le\left((b-1)D\left(1+\frac {|I|}{k}\right)+o(1)\right)(n/k)^{k-1}.
		\label{h:con:eq:deletion-budget}
		\end{align}
		Here $|X_i|=n/k+o(n)$, and
		$\binom n{k-2}=k^{k-2}(n/k)^{k-2}/(k-2)!+O_k(n^{k-3})$.

		Let $r'_w=|\{i\notin I:w\in W_i\}|$.
		By \eqref{h:con:eq:ordinary-own-part}, all but $o(n^{k-1})$ edges of $\mathcal L'_w$
		avoid the part containing $w$.
		They therefore use only the other $k-1$ parts.
		Types with no repeated part contribute at most $1$ after division by $(n/k)^{k-1}$.
		Only the $r'_w$ specified parts can repeat with a positive limit.
		Every other repeated type contributes $o(n^{k-1})$ edges.
		Lemma~\ref{h:con:lem:patterns} gives
		\[
		|\mathcal L'_w|
		\le\left(1+\frac{Q_{k-1}}{k-1}r'_w+o(1)\right)(n/k)^{k-1}.
		\]
		Moreover,
		$\sum_{w\in W} r'_w=\sum_{i\notin I}|W_i|\le(k-|I|)(a+b-1)$.
		Recall from \eqref{h:eq:main-constants} that
		\begin{equation*}
		\beta=\frac{k}{k-1}Q_{k-1}
		=\frac{k}{k-1}
		\left(\frac{(k-1)^{k-1}}{(k-1)!}-1\right).
		\end{equation*}

		There are only $O_{G^{+k}}(n^{k-2})$ edges meeting $W$ at least twice.
		The number of crossing edges in $H-W$ is at most $t^{(k)}(n-|W|)$.
		Using \eqref{h:con:eq:deletion-budget} and \eqref{h:eq:tk-shift},
		we obtain
		\begin{align*}
		e(H)\le t^{(k)}(n)+\biggl[
		\beta\left(1-\frac {|I|}{k}\right)(a+b-1)
		+(b-1)D\left(1+\frac {|I|}{k}\right)+o(1)\biggr](n/k)^{k-1}.
		\end{align*}
		The coefficient apart from $o(1)$ is an affine function of $|I|/k\in[0,1]$.
		It is therefore at most the larger of its values  at $0$ and $1$, 
		\begin{equation*}
		\max\{\beta(a+b-1)+(b-1)D,\ 2(b-1)D\}.
		\end{equation*}
		Both quantities are strictly smaller than $a\alpha$ by
		\eqref{h:eq:admissible}.  This contradicts \eqref{h:con:eq:j-expansion} for large $n$.
	\end{proof}

	Lemma \ref{h:con:lem:B-nonempty} gives $X\ne\varnothing$.
	The set $W$ was chosen using Lemma~\ref{h:con:lem:bounded-W}(i).
	Thus every edge of $H-W$ is crossing.
	We will prove that $|X|=a$ and $W_i=X$ for every $i$.
	Only two further bounds are needed.
	
	\begin{lemma}\label{h:con:lem:Pi-bound}
		For every $i\in[k]$ and $z\in X$,
		$
		|W_i\setminus\{z\}|\le a-1.
		$
		Consequently, $|W_i|\le a$.
		If $|W_i|=a$, then $X\subseteq W_i$.
	\end{lemma}
	\begin{proof}
		Suppose that $w_1,\ldots,w_a$ are distinct vertices of
		$W_i\setminus\{z\}$.
		For each $j\in[a]$, the family $\mathcal L_{w_ji}$ has $\Omega(n^{k-1})$ edges.
		Choose an edge from each family.
		Require the chosen edges to be pairwise disjoint.
		They already avoid every vertex of $W$ except their own $w_j$.
		Each additional forbidden vertex excludes at most $O(n^{k-2})$ choices.
		At most $f$ vertices have been used at each step.
		The choices are therefore possible for sufficiently large $n$.
		
		In each chosen edge, choose two vertices of $X_i$ as matching core endpoints.
		Use the other vertices as private vertices.
		In particular, each $w_j$ is private.
		This is possible because $k\ge3$.
		The resulting $a$ edges form an expansion of $M_a$.
		
		Use $z$ as the star center.
		It has $\Omega(n)$ fat neighbors in $X_i$.
		Choose $b$ distinct such neighbors as leaves, avoiding all selected matching edges.
		All core vertices of $M_a\cup K_{1,b}$ except $z$ now lie in $X_i$. The  core vertex  $z$ has $\Omega(n)$ neighbors in every other part.
		Lemma~\ref{h:con:lem:completion} gives a copy of $G^{+k}$, a contradiction.
		This proves $|W_i\setminus\{z\}|\le a-1$.
		
		Since $X$ is nonempty, $|W_i\setminus\{z\}|\le a-1$ gives $|W_i|\le a$.
		If $|W_i|=a$ and some $z\in X$ were outside $W_i$, then
		$|W_i\setminus\{z\}|=a$.
		This contradicts the same inequality.
	\end{proof}
	
	Fix $i\in[k]$.
	For a tuple $\mathbf x=(x_j)_{j\ne i}\in\prod_{j\ne i}X_j$, write
	$
	D_i(\mathbf x)=X\cap\bigcap_{j\ne i}N_{\mathsf{F}(H)}(x_j).
	$
	This is the set of exceptional vertices adjacent to every vertex of the tuple.
	
	\begin{lemma}\label{h:con:lem:common-neighbors}
		For each $i\in[k]$, all but $o(n^{k-1})$ tuples satisfy
		$|D_i(\mathbf x)|\le a$.
		In particular,
		\begin{equation}\label{h:con:eq:common-sum}
		\sum_{\mathbf x\in\prod_{j\ne i}X_j}|D_i(\mathbf x)|
		\le\bigl(\min\{a,|X|\}+o(1)\bigr)(n/k)^{k-1}.
		\end{equation}
	\end{lemma}
	\begin{proof}
		For $R\subseteq X$ with $|R|=a+1$, let
		$
		C_j(R)=X_j\cap\bigcap_{z\in R}N_{\mathsf{F}(H)}(z).
		$
		We first show that
		\begin{equation}\label{h:con:eq:common-product-small}
		\prod_{j\ne i}|C_j(R)|=o(n^{k-1}).
		\end{equation}
		If this were false, there would be a constant $\varepsilon>0$ and a further
		subsequence on which the product is at least $\varepsilon n^{k-1}$.
		Hence $|C_j(R)|\ge\varepsilon n$ for every $j\ne i$ on that subsequence.
		
		Use one vertex of $R$ as the star center.
		Use each of the other $a$ vertices as one matching core endpoint.
		Every vertex of $R$ has $\Omega(n)$ fat neighbors in $X_i$.
		We can therefore choose distinct vertices of $X_i$ for the remaining $a$ matching vertices
		and the $b$ star leaves.
		All  such core edges  are fat pairs.
		In every other part, the exceptional core vertices have the common
		neighborhood $C_j(R)$ of linear size.
		Lemma~\ref{h:con:lem:completion} gives $G^{+k}$, a contradiction.
		This proves \eqref{h:con:eq:common-product-small}.
		
		If $|D_i(\mathbf x)|\ge a+1$, then some $(a+1)$-set $R\subseteq X$
		is contained in $D_i(\mathbf x)$.
		For fixed $R$, the number of such tuples is exactly
		$\prod_{j\ne i}|C_j(R)|$.
		There are only boundedly many choices of $R$.
		Summing \eqref{h:con:eq:common-product-small} over them proves the first assertion.
		If $|X|\le a$, there are no such sets $R$ and the assertion is immediate.
		
		There are $\prod_{j\ne i}|X_j|=(1+o(1))(n/k)^{k-1}$ tuples in total.
		At most $o(n^{k-1})$ have more than $a$ common neighbors in $X$.
		Each tuple has at most $|X|$ common neighbors in $X$, and $|X|$ is bounded.
		Thus the sum in \eqref{h:con:eq:common-sum} is at most $(a+o(1))(n/k)^{k-1}$.
		The trivial bound on the same sum is $(|X|+o(1))(n/k)^{k-1}$.
		Taking the smaller bound proves the result.
	\end{proof}

	Recall that $Q_k=\frac{k^{k-1}}{(k-1)!}-k$ and  $\alpha=Q_k+k-1$.
	Also $Q_k>0$ because $k\ge3$.
	For example, among $(k-1)$-sets from $k$ large parts, some types use one part twice.
	Their positive coefficients contribute to $Q_k$ in Lemma~\ref{h:con:lem:patterns}.
	
	\begin{lemma}\label{h:con:lem:total-edge-bound}
		For every extremal $G^{+k}$-free $k$-graph $H$, the notation above gives
		\begin{equation}\label{h:con:eq:total-edge-bound}
		e(H)\le t^{(k)}(n)+
		\left[k\min\{a,|X|\}-|X|+
		\frac {Q_k} {k}\sum_{i=1}^k|W_i|+o(1)\right](n/k)^{k-1}.
		\end{equation}
		Consequently, $
		|X|=a$ and $W_i=X$ for every $i\in[k]$.
	\end{lemma}
	\begin{proof}
		We count the edges of $H$ according to how they meet $W$.
		Edges avoiding $W$ are crossing, so there are at most $t^{(k)}(n-|W|)$ of them. Since $W$ is bounded,
		only $O_{G^{+k}}(n^{k-2})$ edges meet $W$ at least twice.
		It remains to count edges in the disjoint families $\mathcal L_w$.
		
		First count those for which the $k-1$ vertices outside $W$ lie in different parts.
		For   $w\in (W\setminus X)\cap V_j$, equation~\eqref{h:con:eq:ordinary-own-part}
		shows that only $o(n^{k-1})$ such edges use $X_j$.
		Every other edge of this type chooses one vertex from each $X_i$ with $i\ne j$.
		There are at most $(1+o(1))(n/k)^{k-1}$ choices.
		The $|W|-|X|$  vertices of $W$ therefore contribute at most
		$(|W|-|X|+o(1))(n/k)^{k-1}$ edges of this type.
		Now consider  $x\in X$.
		By Lemma~\ref{h:lem:fat}(i), only $O_{G^{+k}}(n^{k-2})$ such edges
		contain a non-fat pair between $x$ and a vertex of $V(H)\setminus W$.
		For every remaining edge, its outside vertices form a tuple in
		$\prod_{m\ne i}X_m$ for a unique omitted part $i$, where  $x\in D_i(\mathbf x)$.
		Summing over the $k$ choices of $i$ and applying Lemma~\ref{h:con:lem:common-neighbors}
		bounds this contribution by
		$
		\bigl(k\min\{a,|X|\}+o(1)\bigr)(n/k)^{k-1}.
		$
		
		Next count edges in $\mathcal L_w$ whose outside vertices use some part twice.
		Let $r_w=|\{i:w\in W_i\}|$.
		If $w\notin W_i$, then $|\mathcal L_{wi}|=o(n^{k-1})$.
		After omitting these edges, only the $r_w$ specified parts may repeat.
		Lemma~\ref{h:con:lem:patterns} bounds this contribution by
		$
		\left(\frac {Q_k}{k} r_w+o(1)\right)(n/k)^{k-1}.
		$
		Since $\sum_{w\in W}r_w=\sum_{i\in [k]}|W_i|$, all repeated types contribute at most
		$
		\left(\frac {Q_k}{k}\sum_{i\in [k]}|W_i|+o(1)\right)(n/k)^{k-1}.
		$
		
		Add the bounds and use \eqref{h:eq:tk-shift}.
		This gives \eqref{h:con:eq:total-edge-bound}.
		By Lemma~\ref{h:con:lem:Pi-bound}, each $|W_i|$ is at most $a$.
		Also
		\[
		k\min\{a,|X|\}-|X|=
		\begin{cases}
		(k-1)|X|,&|X|\le a,\\
		ka-|X|,&|X|\ge a.
		\end{cases}
		\]
		This expression is at most $(k-1)a$.
		Equality holds only for $|X|=a$.
		Thus apart from $o(1)$, the coefficient in \eqref{h:con:eq:total-edge-bound} is at most
		$
		(k-1)a+Q_ka=a\alpha.
		$
		The lower bound \eqref{h:con:eq:j-expansion} has this same coefficient.
		The numbers $|X|$ and $|W_i|$ are fixed along the chosen subsequence.
		If $|X|\ne a$ or  some $|W_i|<a$, then $k\min\{a,|X|\}-|X|+
		\frac {Q_k} {k}\sum_{i=1}^k|W_i|<a\alpha -c$ for a  suitable constant $c$ because $Q_k>0$.
		Either case contradicts \eqref{h:con:eq:j-expansion} for large $n$.
		Hence $|X|=a$ and $|W_i|=a$ for every $i$.
		Lemma~\ref{h:con:lem:Pi-bound} now gives $X= W_i$.
	\end{proof}
	
	\subsection{Exact uniqueness}
	\begin{proof}[Proof of Theorem~\ref{h:thm:connected-exact}]
		Define $\widetilde J=J(X;V_1\setminus X,\ldots,V_k\setminus X).$
		Lemma \ref{h:con:lem:total-edge-bound} gives $|X|=a$.
		We first count the edges of $E(H)\setminus E(\widetilde J)$.
		Every edge of $H$ avoiding $W$ is crossing and belongs to $\widetilde J$.
		Every edge meeting $X$ also belongs to $\widetilde J$.
		There are only $O_{G^{+k}}(n^{k-2})$ edges meeting $W$ at least twice.
		
		Consider an edge meeting $W\setminus X$ exactly at one vertex $w\in V_j$.
		Equation~\eqref{h:con:eq:ordinary-own-part} bounds the number of these edges
		that also meet $X_j$ by $o(n^{k-1})$.
		Since $W_i=X$ for every $i$, we have $w\notin W_i$ for every $i$.
		Thus only $o(n^{k-1})$ of these edges use any $X_i$ twice.
		Every remaining edge chooses one vertex from each part other than $V_j$.
		It is crossing and belongs to $\widetilde J$.
		Since $|W|$ is bounded, these  give
		$
		|E(H)\setminus E(\widetilde J)|=o(n^{k-1}).
		$

		Lemma~\ref{h:con:lem:jfree} and the maximal product of balanced part sizes give
		$
		e(H)\ge e(\Ja(n))\ge e(\widetilde J).
		$
		Therefore
		\begin{align}
		|E(\widetilde J)\setminus E(H)|
		&=e(\widetilde J)-e(H)+|E(H)\setminus E(\widetilde J)|\nonumber\\
		&\le |E(H)\setminus E(\widetilde J)|=o(n^{k-1}).
		\label{h:con:eq:fine-closeness}
		\end{align}
		
		All  parts $V_i\setminus X$ of $\widetilde J$ have size $n/k+o(n)$.
		Apply Lemma~\ref{h:con:lem:local}(ii) with $\eta=1/(2k)$.
		Equation~\eqref{h:con:eq:fine-closeness} is smaller than the fixed threshold
		$c_1n^{k-1}$ for all sufficiently large members of the sequence.
		Hence $H\subseteq\widetilde J$.
		It follows that
		$
		e(\Ja(n))\le e(H)\le e(\widetilde J)\le e(\Ja(n)).
		$
		All inequalities are equalities, so $H=\widetilde J$ and $e(\Ja(n))= e(\widetilde J)$. The maximal product of balanced part sizes gives $H\cong\Ja(n)$.
		
		The  part sizes must be balanced.
		Otherwise, two part sizes $u,v$ would satisfy $u\ge v+2$.
		Moving one vertex from the larger part to the smaller would increase their product by
		$(u-1)(v+1)-uv=u-v-1>0$.
		The other part sizes are positive, so this would increase the number of edges.
		This contradicts equality with $e(\Ja(n))$.
		Thus $H\cong\Ja(n)$.
		This contradicts the choice of our sequence.

		The exact extremal formula and uniqueness follow for every sufficiently large $n$.
		The last step needs the error $o(n^{k-1})$ in \eqref{h:con:eq:fine-closeness}.
		An error of size $o(n^k)$ alone would not allow the use of Lemma~\ref{h:con:lem:local}.
		Lemma~\ref{h:con:lem:local}(i) gives the claimed one-edge count.
		Proposition~\ref{h:prop:stability} gives stability, since
		$\chi(G)=k+1$. This completes Theorem~\ref{h:thm:connected-exact}.
	\end{proof}
	
	\section{Hypergraph supersaturation}\label{h:sec:criterion}
	
	For a graph $Q$, let $\rank(Q)$ be the number of edges
	in a spanning forest of $Q$. Set $\rank(\varnothing)=0$.
	Equivalently, this is the largest number of edges in a forest
	contained in $Q$. Thus rank cannot decrease when vertices or edges are added.
	
	The next theorem separates the copy count from the exact extremal
	problem. Its hypotheses will be checked for the full joins and for their extensions.
	
	\begin{theorem}\label{h:thm:criterion}
		Let $k\ge3$, let $G$ be a fixed graph, and let 
		$f=v(G^{+k})$. Suppose that for fixed integers $A\ge1$ and $\rho$,
		the following hold for all sufficiently large $n$:
		\begin{enumerate}[label=(H\arabic*)]
			\item the unique extremal $G^{+k}$-free $k$-graph is $\J_A(n)$;
			\item $c(n,G^{+k})\ge c_*n^{f-A-k}$ for a constant $c_*>0$;
			\item every $U\subseteq V(G)$ with $\chi(G-U)\le k-1$ satisfies
			$\rank(G[U])\ge\rho$.
		\end{enumerate}
		Let $\tau=\rho-A-1$. If $\tau>k-1$, let
		$\gamma=\tau/(\rho-1)$. Then $0<\gamma<1$, and
		there exists a constant $C>0$ such that\begin{equation}\label{h:eq:general-ratio}
		\frac{h_{G^{+k}}(n,q)}{q c(n,G^{+k})}
		\le C\left(\frac{n^{k-1-\tau}}q+
		\frac{q^{\rho-1}}{n^{(k-1)(\rho-1)+\tau}}\right)
		\end{equation}
		uniformly for all sufficiently large $n$ and all integers
		$1\le q\le n^{k-1+\gamma}$.
	\end{theorem}
	
	\begin{proof}
		Start with a balanced complete $k$-partite  graph on
		$X_1,\ldots,X_k$. Split $X_1$ into blocks of size $z$, with
		at most one smaller block. Add all edges within each block.
		Call the resulting graph $P_z$, and let $H_z$ be its $k$-clique
		hypergraph. Thus a $k$-set is an edge of $H_z$ exactly when it
		is a clique in $P_z$.
		
		Suppose $2\le z\le |X_1|/2$, and let
		$|X_1|\ge n/(2k)$. At least $|X_1|/(2z)$ blocks have size
		$z$. They contain at least $n(z-1)/(8k)$ internal pairs.
		For each pair, choose $k-2$ other parts and one vertex in each.
		This gives at least $(k-1)(n/(2k))^{k-2}$ new hyperedges per
		pair. Each has exactly two vertices in $X_1$, so there is no
		double counting. Together with all crossing edges, this gives
		\begin{equation}\label{h:eq:block-edge-count}
		e(H_z)\ge t^{(k)}(n)+d_k(z-1)n^{k-1},
		\qquad d_k=\frac{k-1}{8k(2k)^{k-2}}.
		\end{equation}

		The core of a copy of $G^{+k}$ in $H_z$ is an
		edge-preserving copy of $G$ in $P_z$. Fix the part assigned to
		each core vertex. There are at most $k^{v(G)}$ assignments.
		Let $U$ be the vertices assigned to $X_1$. Since the other parts
		are independent, any possible assignment has
		$\chi(G-U)\le k-1$. Write $r_U=\rank(G[U])\ge\rho$.
		Every connected component of $G[U]$ must lie in a single block.
		Choose one root image in each component in at most $n$ ways.
		Every other image in that component has at most $z$ choices.
		Each core vertex outside $U$ has at most $n$ choices.
		The number of core maps for this assignment is therefore at most
		$
		n^{v(G)-r_U}\cdot z^{r_U}\le n^{v(G)-\rho}\cdot z^\rho.
		$
		The private vertices have at most
		$n^{(k-2)e(G)}$ choices.  Thus
		\begin{equation}\label{h:eq:block-copy-count}
		N_{\F}(H_z)\le k^{v(G)}n^{f-\rho}z^\rho.
		\end{equation}

		By (H1) and \eqref{h:con:eq:j-expansion},
		$
		\ex(n,G^{+k})-t^{(k)}(n)
		\le\binom nk-\binom{n-A}k
		\le\frac{A}{(k-1)!}n^{k-1}.
		$
		Take
		\begin{equation*}
		z=2+\left\lceil d_k^{-1}
		\left(\frac{A}{(k-1)!}+\frac q{n^{k-1}}\right)\right\rceil.
		\end{equation*}
		Then $z=O(1+q/n^{k-1})$. Since
		$\rho-1-\tau=A>0$, we have $0<\gamma<1$.
		Uniformly for $q\le n^{k-1+\gamma}$, the chosen block size is
		$O(n^\gamma)=o(n)$. Hence $z\le|X_1|/2$ for all large $n$.
		Equation~\eqref{h:eq:block-edge-count} gives
		$e(H_z)\ge\ex(n,G^{+k})+q$. Delete edges to reach equality.
		This cannot increase the number of copies.

		Using (H2) and \eqref{h:eq:block-copy-count}, we obtain a constant $C>0$ such that
		\begin{align*}
		\frac{h_{G^{+k}}(n,q)}{q c(n,G^{+k})}
		&\le C\frac{n^{f-\rho}}{q n^{f-A-k}}
		\left(1+\frac q{n^{k-1}}\right)^\rho
		=C\frac{n^{k-1-\tau}}q
		\left(1+\frac q{n^{k-1}}\right)^\rho.
		\end{align*}
		The inequality $(1+x)^\rho\le2^{\rho-1}(1+x^\rho)$ gives
		\eqref{h:eq:general-ratio}.
	\end{proof}

	For a graph $Q$, define
	\begin{equation*}
	r_k(Q)=\min\{\rank(Q[U]):U\subseteq V(Q),\ \chi(Q-U)\le k-1\}.
	\end{equation*}
	The empty graph has chromatic number zero. This parameter measures
	the number of core choices restricted to small blocks in the proof
	of Theorem~\ref{h:thm:criterion}.
	
	\begin{lemma}\label{h:lem:rank-monotone}
		If $Q_1\subseteq Q_2$, then $r_k(Q_1)\le r_k(Q_2)$.
	\end{lemma}
	\begin{proof}
		Take a minimizing set $U$ for $Q_2$ and let $U_1=U\cap V(Q_1)$.
		Then $Q_1-U_1\subseteq Q_2-U$, so $U_1$ is feasible for $Q_1$.
		Also $Q_1[U_1]\subseteq Q_2[U]$, and the rank of a subgraph
		is at most that of the larger graph. Thus
		$r_k(Q_1)\le\rank(Q_1[U_1])\le\rank(Q_2[U])=r_k(Q_2)$.
	\end{proof}
	
	\begin{proposition}\label{h:prop:rank-exact}
		Let $a\ge1$, $b\ge1$, $h\ge0$, and
		$1\le m_1\le\cdots\le m_{k-1}$, with $k\ge3$.
		Let $G=(M_a\sq K_{1,b}\sq I_{h})\vee K_{m_1,\ldots,m_{k-1}}$. Then
		\begin{equation}\label{h:eq:rank-exact}
		r_k(G)=\min\{a+b,\ a+m_1,\ m_1+m_2-1\}.
		\end{equation}
		In particular,  this minimum equals $a+b$ under \eqref{h:eq:admissible}.
	\end{proposition}
	
	\begin{proof}
		Write $L_0=M_a\sq K_{1,b}\sq I_{h}$, and let
		$A_1,\ldots,A_{k-1}$ be the other classes.
		Let $U\subseteq V(G)$ with $\chi(G-U)\le k-1$.
		Let $z$ be the number of classes $A_i$ entirely contained in $U$.
		So
		\[
		\chi(L_0-U)+(k-1-z)=\chi(G-U)\le k-1.
		\]
		If $z=0$, then $L_0-U$ is empty. Hence $G[U]$ contains the
		forest $M_a\sq K_{1,b}$, with rank $a+b$.
		If $z=1$, then $L_0-U$ is independent. Therefore $U\cap V(L_0)$
		is a vertex cover of $L_0$ and has at least $a+1$ vertices.
		Together with the deleted large class, of size at least $m_1$,
		these vertices induce a connected subgraph of order at least
		$a+1+m_1$. Its rank is at least $a+m_1$.
		If $z\ge2$, two deleted classes induce a complete bipartite graph
		whose rank is at least $m_1+m_2-1$.
		This proves the lower bound in \eqref{h:eq:rank-exact}.

		All three values can be obtained by different sets.
		For the first, let $U=V(L_0)$.
		For the second, choose $U=A_1\cup C$, where $C$ consists of the
		star center and one endpoint from each matching edge.
		The set $C$ is independent and a vertex cover of $L_0$.
		Thus $G[U]=K_{m_1,a+1}$, of rank $m_1+a$.
		For the third, choose $U=A_1\cup A_2$.
		The remaining graph has chromatic number $2+(k-3)=k-1$,
		and $G[U]$ has rank $m_1+m_2-1$.
		This proves equality.
		
		Under \eqref{h:eq:admissible},
		$k^{k-1}/(k-1)!>1$ and 
		\begin{align*}
		\max\left\{
		\frac{\beta+D}{\alpha-\beta},
		\frac{2D}{\alpha}
		\right\}\ge\frac{2D}{\alpha}
		=\frac{(k-1)k^{k-1}/(k-1)!}{k^{k-1}/(k-1)!-1}>k-1.
		\end{align*}
		Thus  $a\ge b$.
		Since $m_i\ge a+1$,  $r_k(G)=a+b$.
	\end{proof}
	
	\begin{proof}[Proof of Theorem~\ref{h:thm:connected-main}]
		Theorem~\ref{h:thm:connected-exact} gives hypotheses (H1) and (H2)
		of Theorem~\ref{h:thm:criterion}. Proposition~\ref{h:prop:rank-exact}
		gives (H3) with
		$A=a, \rho=a+b, \tau=\rho-A-1=b-1.$
		The condition $b\ge k+1$ implies $\tau>k-1$.
		Substitution into \eqref{h:eq:general-ratio} gives
		\eqref{h:eq:connected-ratio}, with $\gamma=(b-1)/(a+b-1)$.
	\end{proof}
	\begin{proof}[Proof of Corollary~\ref{h:cor:optimized}]
		If $q\le\delta n^{k-1+\gamma}$,
		the two terms on the right of \eqref{h:eq:connected-ratio} are at most $n^{k-b}$ and
		$\delta^{a+b-1}$. Choose $\delta$ and then $n$ to make each
		contribution at most $1/(2K)$. This proves (i).
		For (ii), the second term tends to zero and the first does so
		uniformly for $q\ge1$. Substituting $q\le n^{k-1+c}$ gives
		(iii),
		where
		$
		\varepsilon=\min\{b-k,\ b-1-c(a+b-1)\}>0.
		$
	\end{proof}
	
	\section{Further extremal results}\label{h:sec:extensions}
	
	The exact theorem also applies to connected subgraphs of a graph $G$ satisfying \eqref{h:eq:admissible}.
	
	\begin{theorem}\label{h:thm:subgraphs}
		Let $G$ satisfy Theorem~\ref{h:thm:connected-exact}.
		Let $Q\subseteq G$ be a connected subgraph containing $a+1$
		pairwise vertex-disjoint copies of $K_{k+1}$.
		The subgraph $Q$ need not be spanning.
		Then $Q^{+k}$ has the same unique extremal hypergraph $\J_a(n)$
		for all sufficiently large $n$. It is  stable and not
		$k$-partite.
	\end{theorem}
	\begin{proof}
		The $a+1$ disjoint cliques in $Q$ give $a+1$ disjoint expanded
		cliques in $Q^{+k}$. Each must use a vertex of the distinguished
		$a$-set in $\J_a(n)$, since its complement is $k$-partite.
		Thus $\J_a(n)$ is $Q^{+k}$-free.
		
		Since $Q\subseteq G$, we have $Q^{+k}\subseteq G^{+k}$.
		Every $Q^{+k}$-free hypergraph is therefore $G^{+k}$-free. Hence
		\[
		e(\J_a(n))\le\ex(n,Q^{+k})
		\le\ex(n,G^{+k})=e(\J_a(n)).
		\]
		Any extremal $Q^{+k}$-free hypergraph is also extremal for
		$G^{+k}$, so uniqueness follows from Theorem~\ref{h:thm:connected-exact}.
		Moreover, $k+1\le\chi(Q)\le\chi(G)=k+1$.
		Proposition~\ref{h:prop:stability} gives stability for $Q^{+k}$.
	\end{proof}
	
	Theorem~\ref{h:thm:subgraphs} allows many edge deletions, but it does not
	assert that all resulting graphs have the same supersaturation range.
	That range also depends on the forest rank studied in the last section.


\begin{thebibliography}{99}
	
	\bibitem{g:Turan}
	P.~Tur\'an,
	On an extremal problem in graph theory,
	\emph{Matematikai\'es Fizikai Lapok} \textbf{48} (1941), 436--452
	(in Hungarian).
	
	\bibitem{g:Erdos1}
	P.~Erd\H{o}s,
	Some recent results on extremal problems in graph theory. Results,
	in: \emph{Theory of Graphs}
	(International Symposium, Rome, 1966),
	Gordon and Breach, New York; Dunod, Paris, 1967, pp.~117--123.
	
	\bibitem{g:Erdos2}
	P.~Erd\H{o}s,
	On some new inequalities concerning extremal properties of graphs,
	in: \emph{Theory of Graphs}
	(Colloquium, Tihany, 1966),
	Academic Press, New York, 1968, pp.~77--81.
	
	\bibitem{g:LovaszSimonovits1975}
	L.~Lov\'asz and M.~Simonovits,
	On the number of complete subgraphs of a graph,
	in: \emph{Proceedings of the Fifth British Combinatorial Conference}
	(Aberdeen, 1975),
	Congressus Numerantium \textbf{15},
	Utilitas Mathematica, Winnipeg, 1976, pp.~431--441.
	
	\bibitem{g:LovaszSimonovits1976}
	L.~Lov\'asz and M.~Simonovits,
	On the number of complete subgraphs of a graph II,
	in: \emph{Studies in Pure Mathematics},
	Birkh\"auser, Basel, 1983, pp.~459--495.
	
	\bibitem{g:Simonovits3}
	M.~Simonovits,
	A method for solving extremal problems in graph theory,
	stability problems,
	in: \emph{Theory of Graphs}
	(Colloquium, Tihany, 1966),
	Academic Press, New York, 1968, pp.~279--319.
	
	\bibitem{g:Mubayi1}
	D.~Mubayi,
	Counting substructures I: Color critical graphs,
	\emph{Advances in Mathematics} \textbf{225} (2010), 2731--2740.
	
	
	\bibitem{h:Mubayi2013}
	D.~Mubayi,
	Counting substructures II: Hypergraphs,
	\emph{Combinatorica} \textbf{33} (2013), 591--612.
	
	
	\bibitem{h:MaYuan2025}
	J.~Ma and L.-T.~Yuan,
	Supersaturation beyond color-critical graphs,
	\emph{Combinatorica} \textbf{45} (2025), 

	\bibitem{g:ChenYuan2026}
	W.~Chen and L.-T.~Yuan,
	Strong counterexamples to a supersaturation question of Ma--Yuan,
	arXiv:2606.09518v1 (2026).
	\url{https://arxiv.org/abs/2606.09518v1}.
	
	\bibitem{h:LLLW2026}
	H.~Li, H.~Liu, X.~Liu and J.~Wang,
	Strong counterexamples to Mubayi's supersaturation conjecture
	in every uniformity,
	arXiv:2606.26735v3 (2026).
	\url{https://arxiv.org/abs/2606.26735v3}.
	
	\bibitem{h:Simonovits1974}
	M.~Simonovits,
	Extremal graph problems with symmetrical extremal graphs.
	Additional chromatic conditions,
	\emph{Discrete Mathematics} \textbf{7} (1974), 349--376.

	\bibitem{h:Gerbner2025}
	D.~Gerbner,
	On non-degenerate Tur\'an problems for expansions,
	\emph{European Journal of Combinatorics} \textbf{124} (2025),
	Article~104071.

	\bibitem{g:Ajtai}
	M.~Ajtai, J.~Koml\'os, M.~Simonovits and E.~Szemer\'edi,
	The solution of the Erd\H{o}s--S\'os conjecture for large trees,
	unpublished manuscript.
	
	\bibitem{h:MaQiu2020}
	J.~Ma and Y.~Qiu,
	Some sharp results on the generalized Tur\'an numbers,
	\emph{European Journal of Combinatorics} \textbf{84} (2020),


	
\end{thebibliography}
\end{document}